\documentclass[12pt]{amsart}

\advance\hoffset by -0.5 truecm

\usepackage{amssymb}
\usepackage{amscd}
\usepackage{verbatim}
\usepackage{epsfig}
\usepackage{euscript}
\usepackage[colorlinks=true]{hyperref}
\hypersetup{urlcolor=blue, citecolor=red}

\newcommand{\Diff}{\operatorname{Diff}}
\newcommand{\Hom}{\operatorname{Hom}}

\newcommand{\supp}{\operatorname{supp}}

\def \R{\mathbb R}
\def \p{\partial}

\def\g{\gamma}
\def \g{\gamma}

\def \<{\langle}
\def \>{\rangle}

\def \tx{\tilde x}
\def \e{\epsilon}
\def \o{\omega}
\def \l{\lambda}
\def \S{\mathbb S}
\def \gxy{\gamma_{x,y,\lambda,\omega}}
\def \xy{x,y,\lambda,\omega}
\def \tM{\widetilde M}
\def \1N{\sum_{k=1}^{N}}
\def \A{\mathcal A}
\def \B{\mathcal B}
\def \mR{\mathbb R}
\newcommand{\abs}[1]{\lvert #1 \rvert}
\newcommand{\eps}{\varepsilon}

\newtheorem{lemma}{Lemma}[section]
\newtheorem{prop}[lemma]{Proposition}
\newtheorem{thm}[lemma]{Theorem}
\newtheorem{cor}[lemma]{Corollary}

\newtheorem*{thm*}{Theorem}
\newtheorem*{prop*}{Proposition}
\newtheorem*{cor*}{Corollary}
\newtheorem*{conj*}{Conjecture}
\numberwithin{equation}{section}
\theoremstyle{remark}
\newtheorem{rem}[lemma]{Remark}
\newtheorem*{rem*}{Remark}
\theoremstyle{definition}

\newtheorem*{Def*}{Definition}

\begin{document}

\title{The geodesic X-ray transform with matrix weights}
\author[G. P. Paternain, M. Salo, G. Uhlmann, H. Zhou]{Gabriel P. Paternain, Mikko Salo, Gunther Uhlmann, Hanming Zhou}
\address{Department of Pure Mathematics and Mathematical Statistics, University of Cambridge, Cambridge CB3 0WB, UK}
\email{g.p.paternain@dpmms.cam.ac.uk}
\address{Department of Mathematics and Statistics, University of Jyv\"askyl\"a, 40014 Jyv\"askyl\"a, Finland}
\email{mikko.j.salo@jyu.fi}
\address{Department of Mathematics, University of Washington, Seattle, WA 98195-4350, USA; Department of Mathematics, University of Helsinki, Helsinki, Finland; HKUST Jockey Club Institute for Advanced Study, HKUST, Clear Water Bay, Kowloon, Hong Kong, China}
\email{gunther@math.washington.edu}
\address{Department of Pure Mathematics and Mathematical Statistics, University of Cambridge, Cambridge CB3 0WB, UK; {\it Current address:} Department of Mathematics, University of California Santa Barbara, Santa Barbara, CA 93106-3080, USA}
\email{hzhou@math.ucsb.edu}

\begin{abstract}
Consider a compact Riemannian manifold of dimension $\geq 3$ with strictly convex boundary, such that the manifold admits a strictly convex function. We show that the attenuated ray transform in the presence of an arbitrary connection and Higgs field is injective modulo the natural obstruction for functions and one-forms. We also show that the connection and the Higgs field are uniquely determined by the scattering relation modulo gauge transformations. 
The proofs involve a reduction to a local result showing that the geodesic X-ray transform with a matrix weight can be inverted locally near a point of strict convexity at the boundary, and a detailed analysis of layer stripping arguments based on strictly convex exhaustion functions. As a somewhat striking corollary, we show that these integral geometry problems can be solved on strictly convex manifolds of dimension $\geq 3$ having non-negative sectional curvature (similar results were known earlier in negative sectional curvature). We also apply our methods to solve some inverse problems in quantum state tomography and polarization tomography.

\end{abstract}

\maketitle

\section{Introduction} \label{sec_introduction}

Let $(M,g)$ be a compact Riemannian manifold with boundary $\partial M$ and dimension $n = \dim(M) \geq 2$. We denote by $SM$ the unit sphere bundle with canonical projection $\pi:SM\to M$ given by $(x,v)\mapsto x$.
Given an invertible matrix-valued function $W\in C^{\infty}(SM;GL(N,\mathbb C))$, we consider the following weighted geodesic X-ray transform of a smooth vector-valued function $h\in C^{\infty}(SM;\mathbb C^{N})$,
\[
(I_Wh)(\g)=\int W(\g(t),\dot{\g}(t))h(\g(t),\dot{\g}(t))\,dt,
\]
where $\gamma$ runs over all unit speed finite length geodesics with endpoints on $\partial M$. 
We shall study the invertibility of $I_{W}$ in two important particular instances:
\begin{enumerate}
\item $W$ arbitrary and $h=f\circ\pi$, where $f\in C^{\infty}(M,\mathbb C^{N})$;
\item the weight $W$ arises from an attenuation given by a $GL(N,\mathbb C)$-connection $A$ and a Higgs field $\Phi$ and $h=f\circ\pi+\alpha$, where $f\in C^{\infty}(M,\mathbb C^{N})$ and $\alpha$ is the restriction of a smooth $\mathbb C^N$-valued 1-form on $M$.
\end{enumerate}

In both cases we shall be able to prove that locally around a boundary point of strict convexity and $n\geq 3$, the ray transform $I_{W}$ is injective (up to the natural kernel in the second case) and we will derive global results under a suitable convexity condition. 

There are several motivations for considering the invertibility of $I_{W}$, but the main one driving the present paper is the following geometric inverse problem: given a $GL(N,\mathbb C)$-connection $A$ on the trivial bundle $M\times \mathbb C^N$, is it possible to determine $A$ from the knowledge of the parallel transport along geodesics with endpoints in $\partial M$? Let us describe this problem in more detail and state right away our main global result.

The geodesics going from $\partial M$ into $M$ can be parametrized by the set 
$$\partial_{\pm} (SM) = \{(x,v) \in SM \,;\, x \in \partial M, \pm\langle v,\nu \rangle \leq 0 \}$$
where $\nu$ is the outer unit normal vector to $\partial M$. For any $(x,v) \in SM$ we let $t \mapsto \gamma(t,x,v)$ be the geodesic starting from $x$ in direction $v$. We assume that $(M,g)$ is {\it nontrapping}, which means that the time $\tau(x,v)$ when the geodesic $\gamma(t,x,v)$ exits $M$ is finite for each $(x,v) \in SM$. 

Let $A$ be a $GL(N,\mathbb C)$-connection, this simply means that $A$ is an $N\times N$ matrix whose entries are smooth 1-forms with values in $\mathbb C$. 

Given a smooth curve $\gamma:[a,b]\to M$, the {\it parallel transport} along
$\gamma$ is obtained by solving the linear differential equation in $\mathbb C^N$:
\begin{equation}
\left\{\begin{array}{ll}
\dot{s}+A_{\gamma(t)}(\dot{\gamma}(t))s=0,\\
s(a)=w\in \mathbb C^N.\\
\end{array}\right.
\label{eq:2}
\end{equation}
The {\it parallel transport} along $\gamma$ is the linear isomorphism $P_{A}(\gamma):\mathbb C^N\to\mathbb C^N$  defined by $P_{A}(\gamma)(w):=s(b)$.
We may also consider the fundamental matrix solution $U:[a,b]\to GL(N,\mathbb C)$ of (\ref{eq:2}).
It solves
\begin{equation}
\left\{\begin{array}{ll}
\dot{U}+A_{\gamma(t)}(\dot{\gamma}(t))U=0,\\
U(a)=\mbox{\rm id}.\\
\end{array}\right.
\label{eq:3}
\end{equation}
Clearly $P_{A}(\gamma)(w)=U(b)w$.

Given $(x,v)\in \partial_{+}(SM)$ we solve (\ref{eq:3}) along $\gamma(t,x,v)$ on the interval
$[a,b]=[0,\tau(x,v)]$ and define a map $C_{A}:\partial_{+}(SM)\to GL(N,\mathbb C)$ by
\[C_{A}(x,v):=U(\tau(x,v)).\] 
We call $C_{A}$ the {\it scattering data} of the connection and by the discussion above, it encapsulates
the parallel transport information along geodesics connecting boundary points.
It is natural to incorporate a potential or Higgs field into the problem by considering a \emph{pair} $(A,\Phi)$, where $A$ is a $GL(N,\mathbb C)$-connection and $\Phi$ is a smooth
map $M\to \mathbb C^{N\times N}$. We can solve a transport equation along geodesics:
\begin{equation*}
\left\{\begin{array}{ll}
\dot{U}+[A_{\gamma(t)}(\dot{\gamma}(t))+\Phi(\gamma(t))]U=0,\\
U(0)=\mbox{\rm id}\\
\end{array}\right.
\end{equation*}
and define scattering data $C_{A,\Phi}$ as above. The inverse problem of recovering the pair $(A,\Phi)$ from $C_{A,\Phi}$ has a natural gauge equivalence:
if $u:M\to GL(N,\mathbb C)$ is smooth  and $u|_{\partial M}=\mbox{\rm id}$ then
\[C_{A,\Phi}=C_{u^{-1}du+u^{-1}Au,u^{-1}\Phi u}.\]

Here is our main global result:

\begin{thm} Let $(M,g)$ be a compact Riemannian manifold of dimension $\geq 3$ with strictly convex boundary, and suppose $(M,g)$ admits a smooth
strictly convex function. 
Let $(A,\Phi)$ and $(B,\Psi)$ be two pairs such that $C_{A,\Phi}=C_{B,\Psi}$.
Then there is a smooth map $u:M\to GL(N,\mathbb C)$ such that $u|_{\partial M}=\mbox{\rm id}$, $B=u^{-1}du+u^{-1}Au$ and $\Psi=u^{-1}\Phi u$.
\label{thm:mainglobal}
\end{thm}

It is easy to see that the existence of a strictly convex function $f$ implies that $(M,g)$ is nontrapping (and hence contractible). Section \ref{foliationcondition} provides
several geometric conditions that imply the existence of a strictly convex function, including in particular
manifolds of non-negative sectional curvature and simply connected manifolds without focal points.
 The class of manifolds of non-negative curvature shows that, in contrast to many earlier results, Theorem \ref{thm:mainglobal} allows for the metric $g$ to have conjugate points.
Consider for instance a complete open manifold $(N,g)$ of positive curvature (e.g.\ a strictly convex hypersurface in Euclidean space, like the paraboloid $z=x^2+y^2$). In \cite{GW} it is shown that such a manifold admits a strictly convex function $f:N\to\mathbb R$ that is an exhaustion, i.e.\ for each $c\in \mathbb R$, $f^{-1}(-\infty,c]$ is a compact subset of $N$. Thus for $c>\inf f$, the manifolds $M=f^{-1}(-\infty,c]\subset N$ provide examples to which Theorem \ref{thm:mainglobal} applies. 
 Given a unit speed geodesic $\gamma:\mathbb{R}\to N$, \cite[Lemma 1]{GM} shows that there is $\tau>0$ such that $\gamma|_{[-\tau,\tau]}$ has index greater than or equal to $n-1$. Hence for $c$ large enough $M$ will always have conjugate points.
 
 It seems worthwhile to state the above case as a separate corollary.

 \begin{cor} 
 Let $(M,g)$ be a compact Riemannian manifold of dimension $\geq 3$ with strictly convex boundary and 
 non-negative sectional curvature.
Let $(A,\Phi)$ and $(B,\Psi)$ be two pairs such that $C_{A,\Phi}=C_{B,\Psi}$.
Then there is a smooth map $u:M\to GL(N,\mathbb C)$ such that $u|_{\partial M}=\mbox{\rm id}$, $B=u^{-1}du+u^{-1}Au$ and $\Psi=u^{-1}\Phi u$.
\end{cor}


Another virtue of Theorem \ref{thm:mainglobal} is that there is no restriction on the pair $(A,\Phi)$. In previous works \cite{FU,Sha,PSU12,GPSU} it was assumed that the structure group was the unitary group. Exceptions are
\cite{E,No,V} but these only deal with flat backgrounds.  The work \cite{GPSU} is the only one that deals with complicated trapped geodesics and topology, but $(M,g)$ must be assumed negatively curved.
One drawback of Theorem \ref{thm:mainglobal} is that it does {\it not} apply to two-dimensional manifolds.

Theorem \ref{thm:mainglobal} is proved by introducing a {\it pseudo-linearization} that reduces the nonlinear problem to a linear one.  This pseudo-linearization already appeared in \cite{PSU12} and it is very similar in spirit to the one used in \cite{SU98, SUV13} for the boundary rigidity problem in which the role of $C_{A,\Phi}$ is played by the scattering relation of $g$. A similar scenario arises in polarization tomography \cite{NS07} and quantum state tomography \cite{Il15} and our local results will have consequences for these two areas as well, see Section \ref{section:furtherapplications}.
 All this points to the fact that the hard result is the linear one and we now proceed to state it. We begin by discussing the local linear problems, since the global ones will follow from them together with a layer stripping argument also using the transport equation.

Let $(M,g)$ be a compact Riemannian manifold with boundary $\partial M$ and dimension $n\geq 3$.
For an open set $O\subset M$, $O\cap \p M\neq \emptyset$, an {\it O-local geodesic} is a unit speed geodesic segment in $O$ with endpoints in $\p M$; we denote the set of these by $\mathcal M_O$. Thus $\mathcal M_O$ is an open subset of the set $\mathcal M$ of all maximal geodesics on $M$. Define the local geodesic X-ray transform with matrix weight $W$ of a function $h$ as the collection $(I_Wh)(\g)$ of integrals of $h$ along $\g\in\mathcal M_O$, i.e.\ as the restriction of the weighted geodesic X-ray transform to $\mathcal M_O$. The local question we wish to consider is the following:

\medskip

\emph{Can we determine $h|_O$, the restriction of $h$ to the open set $O$, from the knowledge of $(I_Wh)(\g)$ for $\g\in\mathcal M_O$?}

\medskip

As mentioned before, in this paper  we focus on the case when $h$ is merely a function on $M$ plus a term induced by a 1-form. We embed $M$ into some neighborhood $\widetilde M$ and extend the metric $g$ smoothly onto $\widetilde M$, so $\gamma\in \mathcal M$ is extended to a smooth geodesic on $\widetilde M$. Let $\rho\in C^{\infty}(\widetilde M)$ be a defining function of $\p M$, so that $\rho>0$ in $M\backslash \p M$;  $\rho<0$ on $\widetilde M\backslash M$, and $\rho$ vanishes non-degenerately at $\p M$. 
 
We show invertibility results for the local weighted geodesic X-ray transform on neighborhoods of a strictly convex boundary point $p\in \p M$ of the form $\{\tx>-c\}$, for sufficiently small $c>0$, where $\tx$ is a function with $\tx(p)=0$, $d\tx(p)=-d\rho(p)$. The level sets of $\tx$ are concave near $p$ relative to the neighborhoods. 

We first consider the case when $h$ is just a function on the base manifold $M$. 

\begin{thm}\label{local function}
Assume $\p M$ is strictly convex at $p\in\p M$. There exists a function $\tx\in C^{\infty}(\widetilde M)$ with $O_p=\{\tx>-c\}\cap M$ for sufficiently small $c>0$, such that $f\in L^2(O_p; \mathbb C^N)$ can be stably determined by the weighted geodesic ray transform $I_W$ restricted to $O_p$-local geodesics in the following sense: for $s\geq 0$, $f\in H^s(O_p;\mathbb C^N)$, the $H^{s-1}$ norm of $f$ restricted to any compact subset of $O_p$ is controlled by the $H^s$ norm of $I_Wf|_{\mathcal M_{O_p}}$.
\end{thm}
 
The control is uniform on compact subsets of $O_p$ that are uniformly away from $\{\tilde x=-c\}$. There is also a reconstruction formula similar to the one in \cite[Theorem 4.15]{SUV14}. It is worth mentioning that Theorem \ref{local function} also works for a general family of curves, see the appendix of \cite{UV15}.

There is a special type of weighted geodesic ray transforms called the {\it attenuated geodesic ray transform}. 
Given a pair $(A,\Phi)$ we are interested in the case when the weight $W$ arises as a solution
of the transport equation on $SM$:
\begin{equation}\label{transport 1}
XW=W\mathcal A,\quad W|_{\p_+SM}=\mbox{\rm id},
\end{equation}
where $X$ is the generating vector field of the geodesic flow and $\mathcal A(x,v):=A_{x}(v)+\Phi(x)$.
Note that even if $\p M$ is strictly convex at $p$, in general solutions to \eqref{transport 1} are only continuous on $SO_p$ (smooth in $SO_p\backslash S(O_p\cap \p M)$), however this will not affect our arguments (cf. Section \ref{one form}).


We denote the weighted geodesic ray transform associated with a pair $(A,\Phi)$ by $I_{\mathcal A}$, so the following corollary is a special case of Theorem \ref{local function}. 

\begin{cor}
Assume $\p M$ is strictly convex at $p\in\p M$. There exists a function $\tx\in C^{\infty}(\widetilde M)$ with $O_p=\{\tx>-c\}\cap M$ for sufficiently small $c>0$, such that the restriction of $f\in L^2(O_p)$ on an arbitrary compact subset of $O_p$ can be stably determined by the local attenuated geodesic ray transform $I_{\mathcal A}f|_{\mathcal M_{O_p}}$.
\end{cor}

In order to prove Theorem \ref{thm:mainglobal} we need to consider functions $h$ which have linear dependence on the velocities, i.e. $h(x,v)=\alpha_{x}(v)+f(x)$, where $\alpha$ is a 1-form.  In this case $I_{\mathcal A}$ automatically exhibits a kernel: if $h=(d+A+\Phi)p$, where $p$ vanishes on the boundary, then
$I_{\mathcal A}(h)=0$. Hence the optimal local theorem is as follows:

\begin{thm}\label{local connection}
Assume $\p M$ is strictly convex at $p\in\p M$. There exists a function $\tx\in C^{\infty}(\widetilde M)$ with $O_p=\{\tx>-c\}\cap M$ for sufficiently small $c>0$, such that for given $h=\alpha+f\in L^2(TO_p;\mathbb C^N) \oplus L^2(O_p;\mathbb C^N)$ with $\alpha$ linear in $v$ there is $p\in H^{1}_{loc}(O_p;\mathbb C^N)$ with $p|_{O_p\cap \p M}=0$ such that $h-(d+A+\Phi)p\in L^2_{loc}(TO_p;\mathbb C^N) \oplus L^2_{loc}(O_p;\mathbb C^N)$ can be stably determined from $I_{\mathcal A}h$ restricted to $O_p$-local geodesics in the following sense: for $s\geq 0$, $h\in H^s(TO_p;\mathbb C^N)\times H^s(O_p;\mathbb C^N)$, the $H^{s-1}$ norm of $h-(d+A+\Phi)p$ restricted to any compact subset of $O_p$ is controlled by the $H^s$ norm of $I_{\mathcal A}h|_{\mathcal M_{O_p}}$.
\end{thm}




As discussed in \cite{UV15, SUV14}, the local uniqueness results can be iterated to obtain global results provided that $(M,g)$ can be foliated by strictly convex hypersurfaces in a suitable way. One contribution of the present paper is a more detailed discussion on conditions that allow this layer stripping argument to work. In particular we observe that the foliation does not need to be adapted to the boundary. This is made precise in the following definitions.

\begin{Def*}
Let $(M,g)$ be a compact manifold with strictly convex boundary.
\begin{enumerate}
\item[(a)] 
$M$ satisfies the \emph{foliation condition} if there is a smooth strictly convex function $f: M \to \mR$.
\item[(b)] 
A connected open subset $U$ of $M$ satisfies the \emph{foliation condition} if there is a smooth strictly convex exhaustion function $f: U \to \mR$, in the sense that the set $\{ x \in U \,;\, f(x) \geq c \}$ is compact for any $c > \inf_U f$.
\end{enumerate}
\end{Def*}

Clearly (a) is a special case of (b). If (b) is satisfied, then $U \cap \partial M \neq \emptyset$, the level sets of $f$ provide a foliation of $U$ by smooth strictly convex hypersurfaces (except possibly at the minimum point of $f$ if $U=M$), and the fact that $f$ is an exhaustion function ensures that the layer stripping can be continued to all of $U$. In Section \ref{foliationcondition} we provide a number of sufficient conditions for (a) or (b) to hold.

We now state the main linear result that will imply Theorem \ref{thm:mainglobal}.

\begin{thm}
Let $(M,g)$ be a compact manifold with strictly convex boundary and $\dim(M) \geq 3$, and let $U$ be a connected open subset of $M$ that satisfies the foliation condition. Let $(A,\Phi)$ be a pair in $U$ and let $h(x,v)=f(x)+\alpha_{x}(v)$ where
$f\in C^{\infty}(U,\mathbb C^N)$ and $\alpha$ is a smooth $\mathbb C^N$-valued 1-form in $U$. If 
\[
(I_{\mathcal A} h)(\gamma) = 0 \ \ \text{for any geodesic $\gamma$ in $U$ with endpoints on $\partial M$},
\]
then 
\[
f = \Phi p \quad \text{and} \quad \alpha = dp + Ap \quad \text{in $U$}
\]
for some $p \in C^{\infty}(U, \mathbb C^N)$ with $p|_{\partial M} = 0$. In particular, if $(M,g)$ admits a smooth strictly convex function, then this result holds with $U = M$.
\label{thm:maingloballinear}
\end{thm}

The proofs of the local theorems use the groundbreaking ideas in \cite{UV15} and further exploited
in \cite{SUV13,SUV14}. The approach to the problem is microlocal and we will set things up so that a suitably localized version of $I_{W}^*I_{W}$ fits into Melrose's scattering calculus \cite{Mel}, after conjugation by an exponential weight.  As in the previous references, a key ingredient is the introduction of an artificial boundary ($\tilde{x}=-c$) which is a little bit less convex than the actual boundary. To obtain the Fredholm property in this calculus, one needs to prove that the boundary symbol is elliptic and this is what we do for the case of invertible matrix weights.
In the case of Theorem \ref{local connection}, ellipticity is achieved in a particular gauge and some care is needed to deal with the pair $(A,\Phi)$, particularly when defining the appropriate localized version of $I_{\mathcal A}^*I_{\mathcal A}$. This is perhaps the most technically challenging aspect of the paper.
The proof of the global theorems combine in a novel way the existence of a strictly convex exhaustion function and a regularity result for the transport equation on nontrapping manifolds with strictly convex boundary.

There is a large literature on geometric inverse problems related to geodesic X-ray transforms, and we mention here further relevant references (mostly for the scalar unweighted case). As discussed above, the present paper follows the microlocal approach leading to local results in dimensions $\geq 3$ initiated in \cite{UV15} and developed in \cite{SUV13,SUV14}. Many earlier results were based on energy estimates pioneered in \cite{Mu} and expounded in \cite{Sh94}; see \cite{PSU1, PSU_hd} and the survey \cite{PSU4} for up-to-date accounts of this method which yields global results on simple manifolds in dimensions $\geq 2$. The recent work \cite{G} extended these results to negatively curved manifolds with nontrivial trapping behavior, using methods from the microlocal analysis of flows. Another method, based on analytic microlocal analysis, has been developed in \cite{SU2, SU3, SU4} and also includes local results on real-analytic simple manifolds \cite{K, KS}.

This paper is organized as follows. Section \ref{foliationcondition} discusses in detail the existence of strictly convex functions; we felt that a thorough exposition was needed to be able to appreciate the global consequences of the local results.
Section \ref{sec_geometric} contains preliminaries necessary for the two subsequent sections. Section \ref{function} and \ref{one form} contain the technical core of the paper and prove Theorems \ref{local function} and \ref{local connection}.
Section \ref{sec_proof_global} contains the proof of Theorem \ref{thm:maingloballinear} stating global injectivity of the attenuated geodesic ray transform on functions plus one-forms. Section \ref{nonlinear connection} discusses the pseudo-linearization and proves also 
the local version of Theorem \ref{thm:mainglobal}. Finally, Section \ref{section:furtherapplications} discusses further applications to quantum state tomography and polarization tomography.

\bigskip

\noindent {\bf Acknowledgements.} 
GPP and HZ were supported by EPSRC grant EP/M023842/1. MS was supported by the Academy of Finland (Finnish Centre of Excellence in Inverse Modelling and Imaging, grant numbers 284715 and 309963) and by the European Research Council under FP7/2007-2013 (ERC StG 307023) and Horizon 2020 (ERC CoG 770924). GU was partly supported by NSF.


\section{Strictly convex functions}
\label{foliationcondition}

In this section we will collect some facts related to strictly convex functions on manifolds. Most of these facts may be found in the literature and many of them are contained in \cite{GulliverLasieckaLittmanTriggiani2004} or \cite{Udriste}, but we will supply some further details. The first main result considers the existence of global smooth strictly convex functions. All manifolds are assumed to be connected and oriented with smooth ($=C^{\infty}$) boundary.

\begin{lemma} \label{lemma_convex_anydimension}
Let $(M,g)$ be a compact manifold with strictly convex boundary and $K$ be the sectional curvature. There is a smooth strictly convex function on $M$ if any one of the following conditions holds:
\begin{enumerate}
\item[(a)]
$M$ is simply connected with $K \leq 0$.
\item[(b)] 
$M$ is simply connected with no focal points.
\item[(c)] 
$K \geq 0$.
\item[(d)] 
$K \geq -\kappa$ where $\kappa > 0$ and $\lambda > \sqrt{\kappa} \,\mathrm{tanh}(\sqrt{\kappa} R)$, where $\lambda$ is the smallest principal curvature of $\partial M$ and $R = \max_{x \in M} d(x,\partial M)$.
\end{enumerate}
Moreover, if $M$ admits a smooth strictly convex function $f$, then 
\begin{itemize}
\item 
$M$ is nontrapping and contractible;
\item 
$M$ contains no closed minimal submanifolds, or more generally there is no smooth harmonic map from a closed manifold into $M$;
\item 
$f$ has a unique local minimum point $x_0$ in $M$ and $f$ attains its global minimum there; and 
\item 
the set of critical points of $f$ is either $\{ x_0 \}$ or the empty set.
\end{itemize}
\end{lemma}

Next we note that in two dimensions there is a complete characterization for the existence of a strictly convex functions, based on using the mean curvature flow \cite{BeteluGulliverLittman2002}. Unfortunately this does not lead to results for the X-ray transform, since the method in this paper fails in two dimensions. The result is also special for two dimensions, since \cite{BangertRoettgen2012} constructs a compact four-dimensional manifold with strictly convex boundary satisfying (b) but not (a).

\begin{lemma} \label{lemma_convex_twodimensions}
Let $(M,g)$ be a two-dimensional compact manifold with strictly convex boundary. The following are equivalent:
\begin{enumerate}
\item[(a)] 
$M$ is nontrapping.
\item[(b)] 
$M$ has no closed geodesic in the interior.
\item[(c)] 
$M$ admits a smooth strictly convex function.
\end{enumerate}
\end{lemma}

We remark that there are results for curvature flows acting on strictly convex hypersurfaces in higher dimensional Riemannian manifolds \cite{Huisken1986}, \cite{Andrews1994}, but these seem to require stronger conditions than those in Lemma \ref{lemma_convex_anydimension}.

We also consider the case where a neighborhood of $M$ admits a suitable strictly convex function. If the boundary is strictly convex, the size of such a neighborhood may be estimated using curvature bounds as in \cite{BuragoZallgaller1988}, \cite{GulliverLasieckaLittmanTriggiani2004}.

\begin{lemma} \label{lemma_convex_neighborhood}
Let $(M,g)$ be a compact manifold with strictly convex boundary. Let $K \geq -\kappa$ where $\kappa > 0$, let $\lambda$ be the smallest principal curvature on $\partial M$, and let $R = \max_{x \in M} \,d(x,\partial M)$.
\begin{enumerate}
\item[(a)]
If $\lambda > \sqrt{\kappa} \,\mathrm{tanh}(\sqrt{\kappa} R)$, there is a smooth strictly convex function in $M$.
\item[(b)]
If $\lambda < \sqrt{\kappa} \,\mathrm{tanh}(\sqrt{\kappa} R)$, there is a smooth strictly convex function $f$ in $M_t$, where 
\[
M_t = \{ x \in M \,;\, d(x,\partial M) < t \}
\]
provided that 
\[
t < \frac{1}{\sqrt{\kappa}} \mathrm{artanh}(\frac{\lambda}{\sqrt{\kappa}}).
\]
The set $\{ x \in M_t \,;\, f(x) \geq c \}$ is compact for any $c > \inf_{x \in M_t} f(x)$.
\end{enumerate}
\end{lemma}

Parts (a) and (b) are related to the foliation conditions in Section \ref{sec_introduction}. The various foliation conditions formulated in \cite{UV15}, \cite{SUV13} in fact reduce to the ones in this paper. For instance:

\begin{lemma} \label{lemma_foliation_general}
Let $(M,g)$ be a compact manifold with smooth boundary, let $a < b$, and assume that $\rho: M \to \mR$ is a smooth function with level sets $\Sigma_t = \rho^{-1}(t)$ so that $\Sigma_t$ is strictly convex when viewed from $\rho^{-1}((a,t])$ and $d\rho|_{\Sigma_t} \neq 0$ for $t \in (a,b]$. Assume also that $\Sigma_b = \partial M$. If 
\[
U = \cup_{t \in (a,b]} \Sigma_t,
\]
then any connected component of $U$ satisfies the foliation condition.
\end{lemma}

Let us now give the proofs of the above results. The ``no focal points'' assumption in Lemma \ref{lemma_convex_anydimension} means that no geodesic segment in $M$ has focal points. It is well known that this implies that $d(\,\cdot\,,p)^2$ for fixed $p$ is strictly convex. For more details on the following equivalent condition and focal points of submanifolds, see \cite{BishopCrittenden1964}, \cite{Eberlein1973}.

\begin{Def*}
We say that $(M,g)$ has no focal points if for any geodesic $\gamma: [0,T] \to M$ and for any nontrivial normal Jacobi field $J(t)$ along $\gamma$ with $J(0) = 0$, one has $\frac{d}{dt} \abs{J(t)}^2 > 0$ for $t  > 0$.
\end{Def*}

\begin{proof}[Proof of Lemma \ref{lemma_convex_anydimension}]
(a) follows from (b): if $K \leq 0$ and if $J$ is a nontrivial normal Jacobi field along a geodesic $\gamma$ with $J(0) = 0$, then $\frac{d}{dt} |J(t)|^2 \big|_{t=0} = 0$ and 
\[
\frac{1}{2} \frac{d^2}{dt^2} |J(t)|^2 = \abs{D_t J}^2 - \langle R(J,\dot{\gamma})\dot{\gamma}, J \rangle \geq \abs{D_t J}^2
\]
showing that $\frac{d^2}{dt^2} |J(t)|^2 \geq 0$ and $\frac{d^2}{dt^2} |J(t)|^2|_{t=0} > 0$. Thus $M$ has no focal points.

(b) Fix any $p \in M$. We will show that $f(x) = \frac{1}{2} d(x,p)^2$ is a smooth strictly convex function in $M$. The manifold $M$ is simply connected with strictly convex boundary and has no conjugate points (since it has no focal points), thus $(M,g)$ is simple and hence diffeomorphic to a closed ball, and $\exp_p$ is a diffeomorphism onto $M$. This is seen as follows: if $(M,g)$ has no conjugate points, $\exp_{p}$ is a surjective covering map (this is proved as in the classical Cartan-Hadamard theorem) and since $M$ is simply connected, $\exp_{p}$ becomes a diffeomorphism.
The domain of $\exp_{p}$ is ball, hence $M$ is a ball too. Alternatively, the injectivity of $\exp_{p}$ may be derived from \cite[Lemma 2.2]{GuM}.

The Hessian of the distance function $r: M \to \mR$, $r(x) = d(x,p)$ may be computed via Jacobi fields, or via the second variation formula (see for instance \cite{Petersen2006}). Let $x \in M \setminus \{p\}$ and $v \in S_x M$ with $v \perp \nabla r$, let $\gamma: [0,T] \to M$ be the geodesic joining $p$ to $x$, and let $J$ be the Jacobi field along $\gamma$ with $J(0) = 0$ and $J(T) = v$. Then 
\begin{align*}
\mathrm{Hess}_x(r)(v,v) = \langle D_t J(T), J(T) \rangle = \frac{1}{2} \frac{d}{dt} |J(t)|^2 \Big|_{t=T}.
\end{align*}
Using the no focal points condition we get $\mathrm{Hess}(r)(v,v) > 0$ whenever $x \neq p$, $v \perp \nabla r$ and $v \neq 0$. Writing $f(x) = h(r(x))$ with $h(t) = \frac{1}{2} t^2$, we have $f \in C^{\infty}(M)$ and 
\[
\mathrm{Hess}(f) = h'(r) \mathrm{Hess}(r) + h''(r) dr \otimes dr.
\]
Since $\mathrm{Hess}(r)(W,\nabla r) = \frac{1}{2} W (\abs{\nabla r}^2) = 0$ for any $W$, we see that $f$ is strictly convex.

(c) This is proved in \cite{Esch86}, and also follows from (d).

(d) This will follow from Lemma \ref{lemma_convex_neighborhood}.

Finally, assume that $(M,g)$ is compact with strictly convex boundary and that $f: M \to \mR$ is a smooth strictly convex function. If $\gamma$ is a unit speed geodesic, then by strict convexity 
\[
\frac{d^2}{dt^2} f(\gamma(t)) = \mathrm{Hess}(f)(\dot{\gamma}(t), \dot{\gamma}(t)) \geq c > 0 \text{ for all $t$}.
\]
Now if $\gamma: [0,\infty) \to M$ would be a trapped geodesic, one would have $f(\gamma(t)) \geq f(\gamma(0)) + \frac{c}{4} t^2$ for $t$ large, contradicting the fact that $f$ is bounded. This shows that $M$ is nontrapping. By \cite{Thorbergsson1978}, any compact nontrapping manifold with strictly convex boundary is contractible. If $\Sigma$ were a closed minimal submanifold contained in $M$, then for any $\Sigma$-geodesic $\eta$ with $\dot{\eta}(0) = Y$ one would have 
\[
\mathrm{Hess}_{\Sigma}(f|_{\Sigma})(Y,Y) = \frac{d^2}{dt^2} f(\eta(t)) \Big|_{t=0} = \mathrm{Hess}(f)(Y,Y) + \langle \nabla f, \nabla_{\dot{\eta}} \dot{\eta} \rangle |_{t=0}.
\]
Since $\eta$ is a $\Sigma$-geodesic, the Gauss formula gives $\nabla_{\dot{\eta}} \dot{\eta}|_{t=0} = \Pi(Y,Y)$ where $\Pi$ is the second fundamental form of $\Sigma$. If $\{ E_1, \ldots, E_k \}$ is a local orthonormal frame of $T \Sigma$, the fact that $\Sigma$ is minimal (i.e.\ $\sum \Pi(E_j,E_j) = 0$) implies that 
\[
\Delta_{\Sigma}(f|_{\Sigma}) = \sum_{j=1}^k \mathrm{Hess}(f)(E_j,E_j) > 0
\]
using strict convexity. Thus $f|_{\Sigma}$ would be subharmonic in the closed manifold $\Sigma$, which is not possible. More generally, if $\Phi: \Sigma \to M$ were a smooth harmonic map from some closed manifold $\Sigma$ into $M$, then $\Phi^* f: \Sigma \to \mR$ would be subharmonic \cite[Proposition 10.4]{Aubin1998}, which is again not possible.

Now let $p, q \in M$ be two distinct local minimum points of $f$. Since $\partial M$ is strictly convex, there is a smooth unit speed geodesic $\gamma: [0,T] \to M$ joining $p$ and $q$ (see for instance \cite{BGS}). The smooth function $h(t) = f(\gamma(t))$ satisfies $h'(0) \geq 0$, $h'(T) \leq 0$ but $h''(t) > 0$, which is a contradiction. Thus $f$ has at most one local minimum point. Since $M$ is compact $f$ attains a global minimum at some $x_0 \in M$. This is the unique local minimum point and the only possible critical point, since by strict convexity any critical point is a local minimum.
\end{proof}

\begin{rem*}
Lemma \ref{lemma_convex_anydimension} (a) and (b) fail without the simply connectedness assumption, since negatively curved strictly convex manifolds with trapped geodesics (such as a piece of the catenoid) do not admit a smooth strictly convex function.
\end{rem*}

\begin{proof}[Proof of Lemma \ref{lemma_convex_twodimensions}]
Clearly (a) implies (b). It is proved in \cite{BeteluGulliverLittman2002} that (b) implies (c). The idea is to construct a strictly convex function by evolving the strictly convex hypersurface $\partial M$ via mean curvature flow. Finally, (c) implies (a) in any dimension as discussed in Lemma \ref{lemma_convex_anydimension}.
\end{proof}

We proceed to Lemma \ref{lemma_convex_neighborhood}. The convex functions there are constructed from the distance function to the boundary. This function is only smooth up to the boundary cut locus, and hence we will need to consider strictly convex functions that are continuous in $M$. We follow \cite{Esch86} and employ convexity in the barrier sense (see \cite{MMU} for relations between the barrier, viscosity and distributional definitions of convexity).

\begin{Def*}
Let $(M,g)$ be a compact manifold. We say that a continuous function $f: M \to \mR$ is $\eps$-convex if for any $x \in M$ and any $\eta < \eps$ there is a support function $f_{x,\eta}$ (that is, a smooth function near $x$ with $f_{x,\eta} \leq f$ and $f_{x,\eta}(x) = f(x)$) such that $\mathrm{Hess}(f_{x,\eta})(v,v) \geq \eta \abs{v}^2$ near $x$. We say that $f: M \to \mR$ is strictly convex if it is $\eps$-convex for some $\eps > 0$.
\end{Def*}

As discussed in \cite{Esch86}, a smooth function is $\eps$-convex if and only if $\mathrm{Hess}(f)(v,v) \geq \eps \abs{v}^2$, and the above notion of strict convexity is the same as in \cite{GW}.

We recall some facts on boundary normal coordinates, see \cite{KKL}, \cite{ItohTanaka2001}, \cite{LiNirenberg2005}. Let $(M,g)$ be a compact manifold with smooth boundary. If $x \in M$, we consider the distance to the boundary 
\[
d(x,\partial M) = \min_{z \in \partial M} d(x,z).
\]
For any $x \in M$, there exists $z \in \partial M$ such that 
\[
x = \gamma_{z,n}(s)
\]
where $s = d(x,\partial M) = d(x,z)$ and $\gamma_{z,n}$ is the normal geodesic starting at $z$ in the inner normal direction $n$ of $\partial M$. If $z \in \partial M$ define the boundary cut value 
\[
\tau_{\partial M}(z) = \sup \{ t  \,;\, \gamma_{z,n}|_{[0,t]} \text{ is the unique shortest geodesic from $\gamma_{z,n}(t)$ to $\partial M$} \}.
\]
Then $\tau_{\partial M}$ is a Lipschitz continuous positive function on $\partial M$. The boundary cut locus is defined by 
\[
\omega_{\partial M} = \{ \gamma_{z,n}(\tau_{\partial M}(z)) \,;\, z \in \partial M \}.
\]
This set has zero measure in $M$, and $d(\,\cdot\,,\partial M)$ is smooth in $M \setminus \omega_{\partial M}$. Observe also that 
\begin{align*}
M \setminus \omega_{\partial M} &= \{ \gamma_{z,n}(t) \,;\, z \in \partial M, \ t \in [0,\tau_{\partial M}(z)) \}, \\
M &= \{ \gamma_{z,n}(t) \,;\, z \in \partial M, \ t \in [0,\tau_{\partial M}(z)] \}
\end{align*}
and that boundary normal coordinates $\exp_{\partial M}^{-1}: (z,t) \mapsto \gamma_{z,n}(t)$ extend as a diffeomorphism onto $M \setminus \omega_{\partial M}$.

\begin{proof}[Proof of Lemma \ref{lemma_convex_neighborhood}]
Write $r(x) = d(x,\partial M)$. Note that $r$ is a distance function, i.e.\ $\lvert \nabla r \rvert = 1$, and $\nabla r(\gamma_{z,n}(t)) = \dot{\gamma}_{z,n}(t)$ for $t < \tau_{\partial M}(z)$. Since $r \geq 0$ in $M$ and $\partial M = \{ r = 0 \}$ has principal curvatures $\geq \lambda > 0$, for $z \in \partial M$ one has 
\[
\mathrm{Hess}(r)_z(v,v) \leq -\lambda \abs{v}^2 \text{ when } \langle v, n(z) \rangle = 0.
\] 
Also, at any point in $M \setminus \omega_{\partial M}$ and for any vector field $W$, we have 
\[
\mathrm{Hess}(r)(W,\nabla r) = \langle \nabla_W \nabla r, \nabla r \rangle = \frac{1}{2} W (\langle \nabla r, \nabla r \rangle) = 0.
\]

Now fix $z \in \partial M$, fix an orthonormal basis $\{ E_1, \ldots, E_{n-1} \}$ of $T_z \partial M$, and parallel transport this basis along $\gamma_{z,n}$ to obtain an orthonormal frame $\{ E_1(t), \ldots, E_n(t) \}$ on $\gamma_{z,n}|_{[0,\tau_{\partial M}(z))}$ with $E_n(t) = \dot{\gamma}_{z,n}(t)$. Consider the symmetric matrix $$A_1(t) = (\mathrm{Hess}(r)_{\gamma_{z,n}(t)}(E_{\alpha}(t), E_{\beta}(t))_{\alpha,\beta=1}^{n-1}.$$
The curvature equation $\nabla_{\partial_r} S + S^2 + R_{\partial_r} = 0$ \cite{Petersen2006}, where $S$ is $\mathrm{Hess}(r)$ written as a $(1,1)$-tensor and $R_{\partial_r}(V,W) = R(V,\partial_r,\partial_r,W)$, implies that 
\[
A_1'(t) + A_1(t)^2 + R_1(t) = 0
\]
in a maximal interval $[0,t_1)$ where $(R_1)_{\alpha \beta}(t) = R(E_{\alpha}(t),\partial_r,\partial_r, E_{\beta}(t))$.
We compare $A_1(t)$ with the solution of 
\begin{align*}
A_2'(t) + A_2(t)^2 + R_2(t) &= 0, \\
A_2(0) &= -\lambda \mathrm{Id}
\end{align*}
in a maximal interval $[0,t_2)$ where $R_2(t) = -\kappa \mathrm{Id}$. Now one has $R_1(t) \geq R_2(t)$ and $A_1(0) \leq A_2(0)$, so the matrix Riccati comparison principle \cite{EH} implies that $t_1 \leq t_2$ and $A_1(t) \leq A_2(t)$ in $[0,t_1)$.

Now one has $A_2(t) = \alpha(t) \mathrm{Id}$ where 
\[
\alpha' + \alpha^2 - \kappa = 0, \quad \alpha(0) = -\lambda.
\]
The solution is given by 
\[
\alpha(t) = \left\{ \begin{array}{cl} \sqrt{\kappa} \,\mathrm{coth}(\sqrt{\kappa}(t-t_0)), & \lambda > \sqrt{\kappa}, \\ -\sqrt{\kappa}, & \lambda = \sqrt{\kappa}, \\ \sqrt{\kappa} \,\mathrm{tanh}(\sqrt{\kappa}(t-t_0)), & \lambda < \sqrt{\kappa} \end{array} \right.
\]
where 
\[
t_0 = \left\{ \begin{array}{cl} \frac{1}{\sqrt{\kappa}} \,\mathrm{artanh}(\frac{\sqrt{\kappa}}{\lambda}), & \lambda > \sqrt{\kappa}, \\ \frac{1}{\sqrt{\kappa}} \,\mathrm{artanh}(\frac{\lambda}{\sqrt{\kappa}}), & \lambda < \sqrt{\kappa}. \end{array} \right.
\]
The maximal time of existence is $t_2 = t_0$ if $\lambda > \sqrt{\kappa}$ and $t_2 = \infty$ otherwise.

Let us assume for now that $\lambda > \sqrt{\kappa} \,\mathrm{tanh}(\sqrt{\kappa} R)$. By the previous discussion one has 
\[
\mathrm{Hess}_x(r)(v,v) \leq \alpha(r(x)) \abs{v}^2 \text{ when } x \in M \setminus \omega_{\partial M}, \ \langle v, \nabla r \rangle = 0.
\]
where $-\infty < \alpha(r) < 0$ for $r \in [0,R]$. To deal with the normal directions we write $f(x) = h(r(x))$ where $h(t) = -t + \frac{t^2}{4R}$. Then we have 
\[
\mathrm{Hess}(f) = h'(r) \mathrm{Hess}(r) + h''(r) dr \otimes dr = (-1 + \frac{r}{2R}) \mathrm{Hess}(r) + \frac{1}{2R} dr \otimes dr.
\]
Since $r \leq R$ and $\mathrm{Hess}(r)(W,\nabla r) = 0$ for any $W$, there is $\eps > 0$ so that 
\[
\mathrm{Hess}(f)_x(v,v) \geq \eps \abs{v}^2 \ \ \text{when } x \in M \setminus \omega_{\partial M}.
\]

It remains to show that $f$ is $\eps$-convex also at any point $x \in \omega_{\partial M}$. We argue as in \cite{Esch86} and construct a support function using a suitable support hypersurface. Let $x = \gamma_{z,n}(t')$ for some $z \in \partial M$ with $t' = \tau_{\partial M}(z) = r(x)$. For any $\mu < \lambda$ let $\bar{S}$ be a small piece of a hypersurface through $z$ with unit normal $\bar{N}$ satisfying $\bar{N}(z) = n(z)$, and with all principal curvatures equal to $\mu$. It follows that 
\[
\langle D_v n, v \rangle \leq -\lambda \abs{v}^2 < -\mu \abs{v}^2 = \langle D_v \bar{N}, v \rangle, \quad v \perp n(z),
\]
and by \cite[Lemma 3.1]{Esch86} the first focal point of $\bar{S}$ along $\gamma_{z,n}$ comes after the first focal point of $\partial M$. Thus if $\bar{S}$ is chosen small enough then $\bar{r} = d(\,\cdot\,,\bar{S})$ is smooth near $x$ with $r(x) = \bar{r}(x)$ and $r \leq \bar{r}$. Writing $\bar{f} = h(\bar{r})$ we obtain a smooth function $\bar{f}$ near $x$ with $\bar{f} \leq f$ and $\bar{f}(x) = f(x)$. Also, repeating the matrix Riccati comparison argument with $\lambda$ replaced by $\mu$ and $\kappa$ replaced by $\kappa' > \kappa$ (related to an extension of $M$ to a larger manifold), we obtain 
\[
\mathrm{Hess}(\bar{f})_x(v,v) \geq \eta \abs{v}^2
\]
where $\eta < \eps$ can be chosen arbitrarily close to $\eps$, by taking $\mu$ close to $\lambda$, $\kappa'$ close to $\kappa$, and $\bar{S}$ small. Thus $f$ is $\eps$-convex in all of $M$.

We have shown that if $\lambda > \sqrt{\kappa} \,\mathrm{tanh}(\sqrt{\kappa} R)$, then $f = -r + \frac{r^2}{4R}$ is a continuous strictly convex function in $M$ where $r = d(\,\cdot\,,\partial M)$. Since $f$ is smooth near $\partial M$, we may extend $f$ to a slightly larger manifold $M'$ so that $f$ remains strictly convex in $M'$. The smoothing procedure of Greene-Wu \cite{GW} then yields a smooth strictly convex function in $M$. Finally, if we assume instead that $\lambda < \sqrt{\kappa} \,\mathrm{tanh}(\sqrt{\kappa} R)$, then we may repeat the above argument in $M_{t_2}$ whenever $t < t_1 < t_2 < \frac{1}{\sqrt{\kappa}} \mathrm{artanh}(\frac{\lambda}{\sqrt{\kappa}})$ to obtain a smooth strictly convex function in $M_{t_1}$, and the restriction $f$ of this function to $M_t$ has the property that $\{ x \in M_t \,;\, f(x) \geq c \}$ is compact for $c > \inf_{x \in M_t} f(x)$.
\end{proof}

\begin{rem*}
Let $(M,g)$ be compact with strictly convex boundary, let $r(x) = d(x,\partial M)$, and let $t > 0$. We note that the strict convexity of the function $f = -r + \frac{r^2}{4R}$ in $\overline{M_t}$ could also be characterized in terms of $\partial M$-Jacobi fields \cite{BishopCrittenden1964}, since the Hessian of $r$ can be expressed in terms of $\partial M$-Jacobi fields similarly as in the proof of Lemma \ref{lemma_convex_anydimension}(b).
\end{rem*}

Finally we consider the case where a subset $U$ of $M$ satisfies the foliation condition in Section \ref{sec_introduction}. The next result gives further information on the associated exhaustion function.

\begin{lemma} \label{lemma_exhaustion_function_properties}
Let $U$ be a connected open subset of $M$, and assume that $f: U \to \mathbb R$ is a smooth strictly convex exhaustion function.
\begin{enumerate}
\item 
If $f(x_0) = \inf f$ for some $x_0 \in U$, then $U = M$, $f$ is a smooth strictly convex function on $M$, and $x_0$ is the unique minimum point and the only possible critical point of $f$.
\item 
If $f$ does not attain its infimum in $U$, then $f$ has no critical points in $U$. Moreover, $f(U) = (a,b]$ where $-\infty \leq a < b < \infty$, $f$ attains its maximum $b$ on (and only on) $\partial M$, and $f(x_j) \to a$ when $x_j \to x$ for any $x \in \partial U \setminus \partial M$. In particular, $U \cap \partial M \neq \emptyset$.
\end{enumerate}
\end{lemma}
\begin{proof}
(1) Let $a = \inf f$, and assume that $f(x_0) = a$ for some $x_0 \in U$. If $x$ is any point in $U$, then $x_0$ and $x$ can be connected by a smooth unit speed geodesic $\gamma: [0,T] \to M$ since $(M,g)$ is strictly convex (see \cite{BGS}). We next show that this geodesic must stay in $U$. For if not, the set $\{ t \in [0,T] \,;\, \gamma(t) \notin U \}$ would be nonempty with infimum $t_0 > 0$, and then $\gamma(t) \in U$ for $t < t_0$. The function $h(t) = f(\gamma(t))$ would satisfy 
\[
h(t) = a + h'(0) t + \frac{1}{2} h''(c(t)) t^2, \qquad t \in [0,t_0),
\]
for some $c(t) \in [0,t]$. But since $f \geq a$ in $U$ one must have $h'(0) \geq 0$, and strict convexity of $f$ implies that $h''(t) > 0$ for $t \in [0,t_0)$. Thus $h(t)$ is strictly increasing in $[0,t_0)$ and for some $\eps > 0$ the set $\{ \gamma(t) \,;\, t \in [\eps,t_0) \}$ is contained in the compact set $\{ x \in U \,;\, f(x) \geq a + c \eps^2 \}$ where $c > 0$, thus $\gamma(t_0) \in U$ which is a contradiction.

We have shown that if $f$ attains its infimum at some $x_0 \in U$, then any point in $U$ can be connected to $x_0$ by a geodesic in $U$. Now let $S_{\eps} = \{ x \in M \,;\, d(x,x_0) < \eps \}$ where $\eps > 0$ is so small that $\overline{S}_{\eps} \subset U$ and normal coordinates at $x_0$ exist near $\overline{S}_{\eps}$. If $x \in U \setminus S_{\eps}$ and if $\gamma: [0,T] \to U$ is a geodesic joining $x_0$ and $x$, the argument in the previous paragraph shows that $f(\gamma(t)) \geq a + c \eps^2$ for $t \in [\eps,T]$. This shows that 
\[
U = \overline{S}_{\eps} \cup \{ x \in U \,;\, f(x) \geq a + c \eps^2 \}.
\]
Thus $U$ is compact as the union of two compact sets. Since $U$ is also open, one must have $U = M$ by connectedness. The remaining conclusions follow from Lemma \ref{lemma_convex_anydimension}.

(2) Now assume that $f$ does not attain its infimum in $U$. Note first that if $x_0 \in U$ were a critical point, then the argument in (1) (with $h'(0) = 0$) would show that $U = M$, and $f$ would necessarily attain its minimum. Thus $f$ has no critical points.

Let $a = \inf f$ and $b = \sup f$, where $a < b$ since otherwise $f$ could not be strictly convex. If $x_j \in U$ satisfy $f(x_j) \to b$, then $\{ x_j \}_{j \geq j_0}$ is contained in the compact set $\{ x \in U \,;\, f(x) \geq c \}$ for some $j_0$ and some $c > a$, hence $(x_j)$ has a subsequence converging to some $x_0 \in U$. Then $f(x_0) = b$, which shows that $b < \infty$. By strict convexity it is not possible for $f$ to attain its maximum in $M^{\mathrm{int}}$, so $f$ can only reach (and does reach) its maximum over $U$ on $\partial M$. By connectedness $f(U)$ is an interval, which must be of the form $(a,b]$ since $f$ does not attain its infimum. Finally, if $x \in \partial U \setminus \partial M$, if $x_j \in U$ satisfy $x_j \to x$, and if $f(x_j)$ does not converge to $a$, then there is $c > a$ so that (after taking a subsequence) $f(x_j) \geq c$ for all $j$. Thus $\{ x_j \}_{j=1}^{\infty}$ is contained in the compact set $\{ x \in U \,;\, f(x) \geq c \}$ which implies that $x \in U$, contradicting the fact that $x \in \partial U \setminus \partial M$.
\end{proof}

\begin{proof}[Proof of Lemma \ref{lemma_foliation_general}]
Given the function $\rho$, we define 
\[
f: U \to \mR, \ \ f(x) = h(\rho(x))
\]
where $h: (a,b] \to \mR$ is a smooth function that will be determined later. Then 
\[
\mathrm{Hess}(f)(w,w) = h'(\rho) \mathrm{Hess}(\rho)(w,w) + h''(\rho) \langle w, \nabla \rho \rangle^2.
\]
Thus, writing $w = v + s N$ where $N = \nabla \rho/\lvert \nabla \rho \rvert$ and $\langle v, N \rangle = 0$, we obtain 
\begin{multline*}
\mathrm{Hess}(f)(w,w) \\
 = h'(\rho) \left[ \mathrm{Hess}(\rho)(v,v) + 2s \mathrm{Hess}(\rho)(v,N) + s^2  \left[ \frac{h''(\rho)}{h'(\rho)} {\lvert \nabla \rho \rvert}^2 + \mathrm{Hess}(\rho)(N, N) \right] \right].
\end{multline*}
The strict convexity of $\Sigma_t$ means that for $x \in \rho^{-1}((a,b])$, 
\[
\mathrm{Hess}(\rho)_x(v,v) > 0 \qquad \text{when } \langle v, N \rangle = 0 \text{ and } v \neq 0.
\]
Let $\lambda_1(x)$ be the smallest eigenvalue of $\mathrm{Hess}(\rho)_x$ restricted to $N(x)^{\perp}$. Since $\rho$ is smooth, $\lambda_1$ is Lipschitz continuous and positive in $\rho^{-1}((a,b])$. It follows that 
\begin{multline*}
\mathrm{Hess}(f)(w,w) \\
 \geq h'(\rho) \left[ \frac{\lambda_1}{2} \lvert v \rvert^2 + s^2  \left[ \frac{h''(\rho)}{h'(\rho)} {\lvert \nabla \rho \rvert}^2 + \mathrm{Hess}(\rho)(N, N) - \frac{2}{\lambda_1} \lvert  \mathrm{Hess}(\rho)(v/\lvert v \rvert, N) \rvert^2 \right] \right].
\end{multline*}
We want to choose $h$ so that $h' > 0$ and the $s^2$ term is positive. More precisely, define 
\[
c(t) = \inf_{x \in \Sigma_t, v \perp N(x)} \lvert \nabla \rho \rvert^{-2} \left[ \mathrm{Hess}(\rho)(N, N) - \frac{2}{\lambda_1} \lvert  \mathrm{Hess}(\rho)(v/\lvert v \rvert, N) \rvert^2 \right]
\]
and note that 
\[
\mathrm{Hess}(f)(w,w) \geq h'(\rho) \frac{\lambda_1}{2} \lvert v \rvert^2 + s^2 \lvert \nabla \rho \rvert^2 \left[ h''(\rho) + c(\rho) h'(\rho) \right].
\]
If $\tilde{c}$ is some smooth function with $\tilde{c} < c$ in $(a,b]$, we choose $h$ as a solution of $h''(t) + \tilde{c}(t) h'(t) = 0$ in $(a,b]$ with $h' > 0$. With this choice of $h$, the function $f$ is smooth and strictly convex in $U$ and $f^{-1}(s) = \Sigma_{h^{-1}(s)}$ for $s \in (h(a), h(b)]$. Let now $U_1$ be any component of $U$. If $c > a$ one has $\{ x \in U \,;\, f(x) \geq c \} = \{ x \in M \,;\, \rho(x) \geq h^{-1}(c) \}$ and the latter set is compact, hence also $\{ x \in U_1 \,;\, f(x) \geq c \}$ is compact and $f|_{U_1}$ is an exhaustion function in $U_1$.
\end{proof}

\section{Geometric setting for the inverse problem}
\label{sec_geometric}

This section provides the necessary geometric preparation for the analysis in Sections \ref{function} and \ref{one form}, see \cite{UV15, SUV13, SUV14} for more details.
 
We start with the defining function $\tilde x$ of the neighborhood introduced in the main theorems. Define the function for $\epsilon>0$ 
$$\tx(z)=-\rho(z)-\e |z-p|^2$$
near $p$, so $\tilde x(p)=0$, here $|\cdot|$ is the Euclidean norm. Observe that since we are dealing with a local problem, one can choose an initial local chart $(U,\varphi)$ near $p$ with $\varphi:U\subset \tM\to \mathbb R^n$, then $|z-p|:=|\varphi(z)-\varphi(p)|$ is well-defined near $p$.
  Then there is $\e_0>0$ such that for $\e\in (0,\e_0)$, 
$$\frac{d^2}{dt^2}(\tilde x\circ\gamma_{p,v})(0)\geq c_0/2$$ at $v\in S_p\widetilde M$ when $\frac{d}{dt}(\tx\circ\gamma_{p,v})(0)=0$. Hence for $c_0>0$ sufficiently small (corresponding to $\e_0$), by shrinking $U$ the function $\tx$ is defined on a neighborhood $U_0$ of $p$ with concave level sets (from the side of the super-level sets of $\tx$), such that for $0\leq c\leq c_0$,
$$O_c=\{\tx>-c\}\cap\{\rho\geq 0\}$$
has compact closure in $U_0\cap\widetilde M$ (Here $O_c$ is exactly the $O_p$ in Theorem \ref{local function}).

We remark that the actual boundary $\rho=0$ only plays a role at the end since ellipticity properties only hold in $U_0$ and the function $h$ we are transforming is supported in $\rho\geq 0$ to ensure localization. Thus for most of the following discussion we completely ignore the actual boundary.

For fixed $c\in [0,c_0]$ we work with $x=\tx+c$, which is the boundary defining function of the region $\{x\geq 0\}$. We complete $x$ to a coordinate system $(x,y)$ on a neighborhood $U_1\subset U_0$ of $p$; we are interested in those $\gamma=\gamma_{z,v}$ with unit tangent vector $v=k(\lambda\p_x+\omega\p_y)$, $k(x,y,\l,\o)>0$, $\omega\in\mathbb S^{n-2}$ and $\lambda$ relatively small.

Now, consider geodesics $\g=\g_{x,y,k\l,k\o}$ with $\g(0)=(x,y),\, \dot{\g}(0)=k(\l\p_x+\o\p_y)$ (where $\omega\in\mathbb S^{n-2}$ and $\lambda$ relatively small) that are parameterized by arclength, as $k$ is a smooth positive function of $(x,y,\lambda,\omega)$ in the region of interest, 
we can totally omit $k$ in the following arguments. From now on we use $(x,y,\l,\o)$, instead of $(x,y,k\l,k\o)$, to parameterize the geodesics $\g$ for sufficiently small $\l$. 

\begin{rem}
Indeed by removing $k$, we are doing a reparametrization of the curve $\g$ and it has (non-constant) speed $1/k$ under the new parameter (we denote the curve under the new parameter by $\tilde{\g}$). Consequently, the weighted geodesic X-ray transform $I_W$ is converted to a new weighted X-ray transform $I_{\tilde W}$ whose weight $\tilde W$ is invertible too. So our argument below is actually proving the invertibility of $I_{\tilde W}$. However, since $(I_Wh)(\g)=(I_{\tilde W}h)(\tilde{\g})$, this implies the invertibility of the original transform $I_W$. 
\end{rem}

Given $(x,y,\lambda,\omega)$ and $\gamma=\gamma_{x,y,\lambda,\omega}$, we define 
$$\kappa(x,y,\lambda,\omega,t):=\frac{1}{2}\frac{d^2}{dt^2}(x\circ\gamma)(t).$$
By the convexity assumption, we have 
$$\kappa(x,y,0,\omega,0)=\frac{1}{2}\frac{d^2}{dt^2}(x\circ\g)(0)>0.$$ 



The basic geometric feature we need is that for $x\geq 0$ and for $\l$ sufficiently small, depending on $x$, the curves $\gamma_{x,y,\l,\o}$ will stay in $O_p$ before exiting from $\p M$. In particular for $x=0$, only $\l=0$ is allowed.

Our inverse problem is recovering $h$ from weighted integrals (we also extend the weight $W$ to an invertible weight on $S\widetilde M$, which is still denoted by $W$)
$$(I_Wh)(\xy)=\int W(\gxy(t), \dot{\g}_{x,y,\l,\o}(t))h(\gxy(t),\dot{\g}_{x,y,\l,\o}(t)) \,dt.$$
Notice that we are interested in $h$ supported in the compact manifold $M$, so the above integral indeed is over a finite interval. However, we may assume the extended metric on $\widetilde M$ is complete, then the above integral is well-defined on $\mathbb R$.

\section{Local invertibility for functions}\label{function}

In this section we will prove Theorem \ref{local function}. The presentation will rely heavily on \cite{UV15}, but we will explain the main points and refer to specific parts of \cite{UV15} whenever needed. A general reference to the microlocal scattering calculus is \cite{Mel}, except that we will use the order convention in \cite[Section 2]{UV15}.

In the following, we will use the notation from Section \ref{sec_geometric}. For $f\in C^{\infty}(M;\mathbb C^N)$, and for a fixed number $F > 0$, we consider the the following operator for $x>0$
\begin{equation}\label{N_F function}
(N_Ff)(x,y)=x^{-2}e^{-F/x}\int_{\R}\int_{\S^{n-2}}W^*(x,y,\l,\o) (I_W e^{F/x}f)(\xy)\chi(\l/x)\, d\l d\o,
\end{equation}
where $W^*$ is the conjugate transpose matrix of $W$ and $\chi$ is a smooth, even, non-negative, compactly supported (for sufficiently small $x$) function on $\mathbb R$. Thus when $x$ is small, only $(I_We^{F/x}f)(x,y,\l,\o)$ with $\l$ sufficiently small (which is exactly the case we expect) will contribute to the operator. We can allow $\chi$ to depend smoothly on $y$ and $\o$; over compact sets such a behavior is uniform. If we show that $N_F$ is invertible as a map between proper spaces of functions supported near $x=0$, we obtain an estimate for $f$ in terms of $I_Wf$.

As mentioned in \cite[Section 3]{UV15}, the maps (notice that we may assume $\widetilde M$ to be complete)
$$\Gamma_+: S\tM\times [0,\infty)\rightarrow [\tM\times\tM; \mbox{diag}],\quad \Gamma_+(z,v,t)=(z, \gamma_{z,v}(t))$$
and
$$\Gamma_-: S\tM\times (-\infty,0]\rightarrow [\tM\times\tM; \mbox{diag}],\quad \Gamma_-(z,v,t)=(z, \gamma_{z,v}(t))$$
are diffeomorphisms near $S\tM\times \{0\}$. Here by denoting $z':=\gamma_{z,v}(t)$, $[\tM\times\tM; \mbox{diag}]$ is the {\it blow-up} of $\widetilde M$ along the diagonal $z=z'$, which essentially means introducing the polar coordinates around the diagonal, so that $\Gamma_\pm(z,v,0)\neq \Gamma_\pm(z,v',0)$ if $v\neq v'$. In particular, for $|t|$ sufficiently small, the local (polar) coordinates 
$$z,|z'-z|, \frac{z'-z}{|z'-z|}$$
are valid on the image of $\Gamma_\pm$, where $|\cdot|$ is the Euclidean norm.
 
Recall the local coordinates $(x,y)$ near the strictly convex boundary point $p$, we write $z=(x,y)$ and $z'=(x',y')$, then similar to \cite{UV15}, it is  convenient to use 
\begin{equation}\label{polar coordinates}
x,y,|y'-y|,\frac{x'-x}{|y'-y|},\frac{y'-y}{|y'-y|}
\end{equation}
 as the local coordinates on the image of $\Gamma_\pm$ for $|t|$ small,  
 when $|y'-y|$ is large relative to $|x'-x|$, i.e. in our region of interest.

On the other hand, our analysis is carried out on the region $\Omega=\{x\geq 0\}$, which can be viewed as a manifold with boundary $x=0$. The operator $N_F$, which is defined in the interior of $\Omega$ at this point, will act on functions that could be non-zero near the boundary $x=0$, 
thus the standard pseudodifferential calculus does not work. This is where the scattering calculus comes in.

$\Omega\times \Omega$ is a manifold with corner $\p\Omega\times\p\Omega$, in addition to blowing up the diagonal, one also blows up the corner. In particular, the blow-up of the intersection of the diagonal with (the interior of) the front face of the blown up corner gives the {\it scattering front face} (of Melrose's scattering double space).
Near the interior of the scattering front face, there are natural scattering coordinates (see also \cite[Section 2]{UV15}) 
$$x,y, X=\frac{x'-x}{x^2}, \, Y=\frac{y'-y}{x},$$
(so $x=0$ defines the scattering front face) under which \eqref{polar coordinates} becomes
\begin{equation}\label{scattering coordinate}
x,y,x|Y|,\frac{xX}{|Y|},\hat Y
\end{equation}
with $\hat Y=\frac{Y}{|Y|}$. 

Given a curve $\gamma=\gamma_{x,y,\lambda,\omega}$ with $(\gamma,\dot\gamma)=(x',y',\lambda',\omega')$, we have near $t=0$
\begin{equation}\label{asymptotic1}
\begin{split}
x'=x+\lambda t+\kappa t^2+\mathcal O(t^3),&\quad y'=y+\omega t+\mathcal O(t^2),\\
\lambda'=\lambda+2\kappa t+\mathcal O(t^2),&\quad \omega'=\omega+\mathcal O(t),
\end{split}
\end{equation}
recall that $\kappa=\kappa(x,y,\lambda,\omega,t)$ defined in Section 3 is proportional to the second derivative of $x'$ with respect to $t$. By \eqref{asymptotic1} and the diffeomorphisms $\Gamma_\pm$, we have near $t=0$,
\begin{equation*}
\begin{split}
 & \hspace{1in} t\circ \Gamma_\pm^{-1}=\pm |y'-y|+\mathcal O(|y'-y|^2),\\
\lambda\circ \Gamma_\pm^{-1}=\pm &\frac{(x'-x)-\kappa |y'-y|^2}{|y'-y|} +\mathcal O(|y'-y|^2),\quad \omega\circ \Gamma_\pm^{-1}=\pm \frac{y'-y}{|y'-y|}+\mathcal O(|y'-y|),\\
\lambda'\circ \Gamma^{-1}_\pm=\pm & \frac{(x'-x)+\kappa |y'-y|^2}{|y'-y|} +\mathcal O(|y'-y|^2),\quad \omega'\circ \Gamma_\pm^{-1}=\pm \frac{y'-y}{|y'-y|}+\mathcal O(|y'-y|).
\end{split}
\end{equation*}
Then applying \eqref{scattering coordinate}, 
\begin{equation}\label{asymptic2}
\begin{split}
& \quad t\circ \Gamma_\pm^{-1}=\pm x|Y|+\mathcal O(x^2),\\
\lambda\circ \Gamma_\pm^{-1}=\pm x&\frac{X-\kappa |Y|^2}{|Y|} +\mathcal O(x^2),\quad \omega\circ \Gamma_\pm^{-1}=\pm \hat Y+\mathcal O(x),\\
\lambda'\circ \Gamma_\pm^{-1}=\pm x&\frac{X+\kappa |Y|^2}{|Y|} +\mathcal O(x^2),\quad \omega'\circ \Gamma_\pm^{-1}=\pm \hat Y+\mathcal O(x).
\end{split}
\end{equation}
The coefficients in the remainder terms are all smooth under the coordinates \eqref{scattering coordinate}. Thus
\begin{equation}\label{coordinates change}
(\Gamma_\pm^{-1})^*dt\, d\lambda\, d\omega=x^2|Y|^{1-n}J_\pm(x,y,x|Y|, \frac{xX}{|Y|},\hat Y)\,dX dY
\end{equation}
with smooth positive density functions $J_\pm$, $J_\pm|_{x=0}=1$.



Now by \eqref{asymptic2} and \eqref{coordinates change}, we obtain the Schwartz kernel of the integral operator $N_F$
\begin{equation}\label{kernel on function}
\begin{split}
K_F(x,y,X,Y) =\sum_{\bullet=+, -}  & e^{-\frac{FX}{1+xX}}|Y|^{-n+1}\chi(\frac{\lambda\circ \Gamma^{-1}_\bullet}{x})W^*(x,y,\lambda\circ\Gamma^{-1}_\bullet,\omega\circ\Gamma^{-1}_\bullet)\\
& W(x+x^2X, y+xY,\lambda'\circ\Gamma^{-1}_\bullet, \omega'\circ\Gamma^{-1}_\bullet)J_\bullet (x,y,x|Y|,\frac{xX}{|Y|},\hat Y).
\end{split}
\end{equation}
Note that on the blow-up of the scattering diagonal (or the {\it lifted diagonal}), $\{X=0,Y=0\}$, in the region $|Y|>\epsilon |X|$, thus on the support of $\chi$,
\begin{equation*}
x,y,|Y|,\frac{X}{|Y|},\hat Y
\end{equation*}
are valid coordinates. In particular, $|Y|$ becomes the defining function of the front face of the lifted diagonal. 
Similar to the discussion after \cite[equation (3.16)]{UV15} one can check that $\chi$, $W$ and $W^*$ are smooth on the blow-up space, then due to the exponential weight, the Schwartz kernel $K_F$ is smooth in $(x,y)$ up to the scattering front face $x=0$ with values in functions Schwartz in $(X,Y)$ for $(X,Y)\neq 0$, conormal to the diagonal $(X,Y)=0$.
Thus $N_F$ is a scattering pseudodifferential operator of order $(-1,0)$ on $\Omega$, i.e. $N_F\in \Psi_{sc}^{-1,0}(\Omega)$, in the sense of Melrose's scattering calculus, see \cite[Section 2]{UV15} and \cite{Mel}.


Let $^{sc}T^*\Omega$ be the scattering cotangent bundle over $\Omega$ whose basis is $\{\frac{dx}{x^2}, \frac{dy}{x}\}$ with $\frac{dy}{x}$ a short notation for $(\frac{dy_1}{x},\cdots,\frac{dy_{n-1}}{x})$. We denote the dual space, the scattering tangent bundle, by $^{sc}T\Omega$ with the basis $\{x^2\p_x, x\p_y\}$. We are interested in the asymptotic behavior of the principal symbol of $N_F$ acting on $(z,\zeta)\in\, ^{sc}T^*\Omega$, $z=(x,y)$. We will show that $N_F$, as a scattering pseudodifferential operator, is (fully) elliptic in the sense that its principal symbol is positive definite at both the {\em fiber infinity} of $^{sc}T^*\Omega$, i.e. when $|\xi|$ is sufficiently large, and the {\em finite points} of $^{sc}T^*\Omega$, in particular when $x$ is sufficiently close to $0$, while $|\xi|$ is relatively small comparing with $x^{-1}$. 
Since the Schwartz kernel $K_F$ is smooth in $(x, y)$ down to $x=0$, it suffices to investigate the principal symbol at $x=0$. Once we show the full ellipticity at $x=0$, by smoothness on $x$, the same result holds in an open neighborhood of $\overline O=\Omega\cap M$ in $\Omega$ assuming that $c>0$ is small enough.


Notice that the two terms with respect to $\Gamma_+$ and $\Gamma_-$ in \eqref{kernel on function} have the same contribution to the Fourier transform of $K_F$, so we can ignore $\Gamma_-$. We have thus proved the following result.

\begin{lemma}\label{boundary symbol for function}
The Schwartz kernel of $N_F$ in the interior of the scattering front face $x=0$ equals 
\begin{equation}\label{boundary}
K_F^b(y,X,Y)= e^{-FX}|Y|^{-n+1}(W^*\,W)(0,y,0,\hat Y)\chi\left(\frac{X-\kappa(0,y,0,\hat Y)|Y|^2}{|Y|},y,\hat{Y}\right).
\end{equation}
\end{lemma}

From now on we denote $(X-\kappa|Y|^2)/|Y|$ by $S$ and $X/|Y|$ by $\tilde S$.

We denote the principal symbol of $N_F$ at fiber infinity by $\sigma_p(N_F)$ and the principal symbol at finite points by $\sigma_{sc}(N_F)$.

\begin{prop}\label{interior elliptic}
Suppose $\chi\in C^{\infty}_c(\R)$, $\chi\geq 0$ with $\chi>0$ near $0$, then $N_F$ is elliptic at fiber infinity of $^{sc}T^*\Omega$.
\end{prop}


\begin{proof}
To study the principal symbol at fiber infinity, we need to evaluate the Fourier transform of the Schwartz kernel in $(X,Y)$ as $|\zeta|\to \infty$ with $\zeta=(\xi,\eta)$ the Fourier dual variables of $(X,Y)$, see also the proof of \cite[Proposition 3.3]{UV15}. In particular it is of conormal singularity at the diagonal $(X,Y)=0$. Changing the coordinates $(X,Y)\to (|Y|,\tilde S,\hat Y)$, in view of the compact support of $\chi$, the leading order behavior of the Fourier transform as $|\zeta|\to \infty$ is encoded in the integration of the restriction of the Schwartz kernel to the front face of the blow-up of the diagonal (by inserting a cutoff function in $|Y|$ concentrating around $0$ without affecting the principal symbol) along the orthogonal equatorial sphere corresponding to $\zeta$, i.e. those $(\tilde S,\hat Y)$ with $\xi\tilde S+\eta\cdot\hat Y=0$. Notice that the singular factor $|Y|^{-n+1}$ in the kernel is canceled after the coordinates change. Moreover, since the front face of the lifted diagonal is defined by $|Y|=0$, the extra vanishing factor $\kappa|Y|$ in $S=\tilde S-\kappa|Y|$ can be dropped at the front face.

The above analysis implies the principal symbol is essentially the following matrix
\begin{equation}\label{function standard symbol}
|\zeta|^{-1}\int_{\zeta^{\perp}\cap(\mathbb R\times \mathbb S^{n-2})}(W^*W)(0,y,0,\hat Y)\chi(\tilde S)\,d\tilde S d\hat Y.
\end{equation}
Since $W$ is invertible, $W^*\,W\,\chi$ is a positive semi-definite Hermitian matrix-valued function. 
In particular, under our assumption on $\chi$ and $n>2$ we can always find some $(\tilde S, \hat Y)$ orthogonal to $\zeta$ such that $\chi(\tilde S)>0$, i.e. there exists $(\tilde S, \hat Y)$ orthogonal to $\zeta$ such that for any non-zero $f=(f_1,\cdots,f_N)^T$,
$$\Big(\sigma_p(N_F)f,f\Big)=\int_{\zeta^{\perp}\cap(\mathbb R\times \mathbb S^{n-2})}\Big|W(0,y,0,\hat Y)f\Big|^2\chi(\tilde S)\,d\tilde S d\hat Y>0,$$
in other words the principal symbol $\sigma_p(N_F)$ is positive definite. Therefore the operator $N_F$ is elliptic at fiber infinity.
\end{proof}

Then we study the principal symbol of $N_F$ at finite points of $^{sc}T^*\Omega$. The exponential weights and properly chosen cut-off function $\chi$ can help us eliminate possible issues of the principal symbol near the zero section of $^{sc}T^*\Omega$. In particular, we show that 
the boundary principal symbol matrix is 
bounded below in absolute value by $c\<(\xi,\eta)\>^{-1}$, some $c>0$.  

\begin{prop}\label{boundary elliptic}
For $F>0$ there exists $\chi\in C^{\infty}_c(\R)$, $\chi\geq 0, \chi(0)=1$, such that $N_F$ is elliptic at finite points of $^{sc}T^*\Omega$.
\end{prop}
\begin{proof}
In order to find a suitable $\chi$ to make $N_F$ elliptic at finite points, we make use of the strategy of \cite{UV15} and its appendix, namely we first do the calculation for $\chi(s)=e^{-s^2/(2F^{-1}\kappa)}$ with $F>0$ (here we need the positivity of $\kappa$), so $\hat\chi(\cdot)=c\sqrt{F^{-1}\kappa}e^{-F^{-1}\kappa|\cdot|^2/2}$ for appropriate $c>0$. 
The Fourier transform of $K_F^b$ in $X$, $\mathcal{F}_XK_F^b(y,\xi,Y)$, is a non-zero multiple of 
\begin{equation}\label{FX}
|Y|^{2-n}(W^*W)(0,y,0,\hat Y)(F^{-1}\kappa)^{\frac{1}{2}}e^{-F^{-1}(\xi^2+F^2)\kappa|Y|^2/2},
\end{equation}
where $\kappa=\kappa(0,y,0,\hat Y)$.

Now we use polar coordinates to compute the Fourier transform of \eqref{FX} in $Y$ up to some non-zero multiple. We denote $\frac{F^{-1}(\xi^2+F^2)}{2}$ by $b$, then
\begin{align*}
& \mathcal{F}_Y(\mathcal{F}_X K_F^b)(y,\xi,\eta)\\
\simeq & \int e^{-i\eta\cdot Y}|Y|^{2-n}(W^*W)(0,y,0,\hat Y)\kappa^{1/2}e^{-b\kappa|Y|^2}\, dY\\
= & \int_0^{+\infty}\int_{\S^{n-2}} e^{-i\eta\cdot\hat{Y}|Y|}|Y|^{2-n}(W^*W)(0,y,0,\hat Y)\kappa^{1/2}e^{-b\kappa|Y|^2}|Y|^{n-2}\,d|Y|d\hat Y\\
\simeq & \int_{\S^{n-2}} b^{-1/2}(W^*W)(0,y,0,\hat Y)e^{-|\eta\cdot\hat Y|^2/4b\kappa}\,d\hat Y,
\end{align*}
here $\simeq$ means equal up to some multiple.

We denote $(\xi^2+F^2)^{1/2}$ by $\<\xi\>$, then $\sigma_{sc}(N_F)(y,\xi,\eta)$ is a constant multiple of
$$\<\xi\>^{-1}\int_ {\S^{n-2}} (W^*W)(0,y,0,\hat Y)e^{-|\frac{\eta}{\<\xi\>}\cdot\hat Y|^2/2F^{-1}\kappa(0,y,0,\hat Y)}\,d\hat Y.$$
Under our assumption $\kappa(0,y,0,\hat Y)>0$, thus there exist positive $c_1, c_2$ that depend on $y$ and are locally uniform such that $0< c_1\leq \kappa\leq c_2$. We denote the smallest eigenvalue of $W^*W$ by $\sigma_{min}$, since $W^*W$ is a positive definite Hermitian matrix, similar bounds hold for $\sigma_{min}$. It is easy to see that the symbol is elliptic. Moreover, we study the joint $(\xi,\eta)$-behavior 
of $\sigma_{sc}(N_F)(y,\xi,\eta)$.  

When $\frac{|\eta|}{\<\xi\>}$ is bounded from above, then $\<\xi\>^{-1}$ is equivalent to $\<(\xi,\eta)\>^{-1}$ in this region in terms of decay rates, 
\begin{equation*}
\begin{split}
& \bigg((\int_{\S^{n-2}} W^*W(0,y,0,\hat Y)e^{-|\frac{\eta}{\<\xi\>}\cdot\hat Y|^2/2F^{-1}\kappa(0,y,0,\hat Y)}\,d\hat Y)f,\,f\bigg)\\
\geq & \int_{\S^{n-2}}\big(\sigma_{min}(0,y,0,\hat Y)|f|^2\big) e^{-c|\frac{\eta}{\<\xi\>}|^2} \,d\hat Y\\
\geq & \bigg(\int_{\S^{n-2}}C'e^{-C''} \,d\hat Y\bigg)|f|^2=C|f|^2.
\end{split}
\end{equation*}
Thus $(\sigma_{sc}(N_F)(y,\xi,\eta)f,f)\geq C\<\xi\>^{-1}|f|^2\simeq C\<(\xi,\eta)\>^{-1}|f|^2$. 

When $\frac{|\eta|}{\<\xi\>}$ is bounded from below, in which case $\<(\xi,\eta)\>^{-1}$ is equivalent to $|\eta|^{-1}$ in terms of decay rates, we write $\hat Y=(\hat Y^{\parallel}, \hat Y^{\perp})$ according to the orthogonal decomposition of $\hat{Y}$ relative to $\frac{\eta}{|\eta|}$, where $\hat Y^{\parallel}=\hat Y\cdot \frac{\eta}{|\eta|}$, and $d\hat Y$ is of the form $a(\hat Y^{\parallel})d\hat Y^{\parallel}d\theta$ with $\theta=\frac{\hat Y^{\perp}}{|\hat Y^{\perp}|}, a(0)=1$. Then
\begin{equation}
\begin{split}
& \bigg((\int_{\S^{n-2}} W^*W(0,y,0,\hat Y)e^{-|\frac{\eta}{\<\xi\>}\cdot\hat Y|^2/2F^{-1}\kappa(0,y,0,\hat Y)}\,d\hat Y)f,\,f\bigg)\\
\geq & \int_{\S^{n-2}} \big(\sigma_{min}(0,y,0,\hat Y)|f|^2\big)e^{-c|\frac{\eta}{\<\xi\>}\cdot\hat Y|^2}\,d\hat Y\\
 \geq & C'\bigg(\int_{\R}\int_{\S^{n-3}} e^{-c(\hat Y^{\parallel}\frac{|\eta|}{\<\xi\>})^2} a(\hat Y^{\parallel})\,d\theta d\hat Y^{\parallel}\bigg)|f|^2\\
= & C'\bigg(\frac{\<\xi\>}{|\eta|}\int_{\R}\left\{\frac{|\eta|}{\<\xi\>}e^{-c(\hat Y^{\parallel}\frac{|\eta|}{\<\xi\>})^2}\right\}a(\hat Y^{\parallel})\,d\hat Y^{\parallel}\int_{\S^{n-3}} \,d\theta\bigg)|f|^2.\label{Dirac}
\end{split}
\end{equation}
Since $\frac{|\eta|}{\<\xi\>}e^{-c(\hat Y^{\parallel}\frac{|\eta|}{\<\xi\>})^2}\rightarrow \delta_0$ in distributions as $\frac{|\eta|}{\<\xi\>}\rightarrow\infty$, the last term in \eqref{Dirac} is equal to $C'\frac{\<\xi\>}{|\eta|}\int_{\S^{n-3}} d\theta=2C\frac{\<\xi\>}{|\eta|}$ (for some $C>0$) modulo terms decaying faster as $\frac{|\eta|}{\<\xi\>}\rightarrow\infty$. In particular, there is $K>0$, such that 
$$C'\int_{\R}\left\{\frac{|\eta|}{\<\xi\>}e^{-c(\hat Y^{\parallel}\frac{|\eta|}{\<\xi\>})^2}\right\}a(\hat Y^{\parallel})\,d\hat Y^{\parallel}\int_{\S^{n-3}} \,d\theta \geq C$$
for $\frac{|\eta|}{\<\xi\>}\geq K$. (Note that the integral on $\S^{n-3}$ uses very strongly the assumption $n\geq 3$; when $n=3$, $d\theta$ is the point measure.) Thus $(\sigma_{sc}(N_F)(y,\xi,\eta)f,f)\geq C\frac{1}{\<\xi\>}\frac{\<\xi\>}{|\eta|}|f|^2=C|\eta|^{-1}|f|^2\simeq C\<(\xi,\eta)\>^{-1}|f|^2$.

Therefore we deduce that $\sigma_{sc}(N_F)(y,\xi,\eta)$ is bounded below by $c\<(\xi,\eta)\>^{-1}$ for some $c>0$, i.e. the ellipticity claim for the case that $\chi$ is a Gaussian. Now we pick a sequence ${\chi_k}\in C_c^{\infty}(\R)$ which converges to the Gaussian in Schwartz functions, then the Fourier transforms $\hat{\chi}_k$ converge to $\hat{\chi}$. One concludes that for some large enough $k$, if we use $\chi_k$ to define the operator $N_F$, then the Fourier transform of $K_F^b$, i.e. the boundary principal symbol, still has lower bounds $c'\<(\xi,\eta)\>^{-1}, \, c'>0$, as desired.
\end{proof}

The full ellipticity implies Fredholm properties of $N_F$ between appropriate Sobolev spaces $H^{s,r}_{sc}(\Omega)$ corresponding to the scattering structure, see \cite[Sections 2.1, 2.2]{UV15}. As discussed in \cite[Section 2.4]{UV15}, for our case, $N_F$ is fully elliptic in an open neighborhood $U$ of $\overline O$ in $\Omega$, then there is a local parametrix $P\in \Psi^{1,0}_{sc}(\Omega)$ of $N_F$ such that $PN_F=Id+R$ with $R\in \Psi^{0,0}_{sc}(\Omega)$, locally smoothing (i.e. in $\Psi^{-\infty,-\infty}_{sc}$) over some open subset $U'$ of $U$ with compact closure, containing $\overline O$.

\begin{proof}[Proof of Theorem \ref{local function}] By Proposition \ref{interior elliptic} and Proposition \ref{boundary elliptic}, we have that $N_F\in\Psi_{sc}^{-1,0}$ is elliptic both in the sense of the standard principal symbol (fiber infinity), and the scattering principal symbol (finite points). Then we can apply the argument of \cite[Section 3.7]{UV15}. Concretely, by the Fredholm properties of $N_F$, the space of elements of the kernel of $N_F$ supported in $\overline O$, $\mbox{Ker}\, N_F\cap \{f\in H^{s,r}_{sc}(\Omega): \supp f\subset \overline O\}$, is of finite dimension. The result is uniform in $c$ small, by the arguments in \cite[Section 2.5]{UV15}, the operator $PN_F=Id+R$ is indeed invertible on elements of the kernel of $N_F$ supported in $\overline O$ for sufficiently small $c$, i.e. this subspace of the kernel of $N_F$ is actually trivial, and further the following stability estimate holds on the space $\{f\in H^{s,r}_{sc}(\Omega): \supp f\subset \overline O\}$
\begin{equation*}
\|f\|_{H^{s,r}_{sc}(\Omega)}\leq C\|N_F f\|_{H^{s+1,r}_{sc}(\Omega)}.
\end{equation*}
Finally, when $\supp f$ is uniformly away from $x=0$, so $\supp f\cap \p\Omega=\emptyset$, the weighted Sobolev norms in the above estimate can be replaced by the usual Sobolev norm (i.e. $H^s(O)$), thus in view of \eqref{N_F function} (taking $x\geq c'>0$ for some $c'<c$, so due to the cutoff $\chi$ the integration in $\lambda$ in \eqref{N_F function} is on a finite interval)
\begin{equation*}
\|f\|_{H^{s}(O)}\leq C_1\|e^{-F/x} f\|_{H^s(O)}\leq C_2\|N_F e^{-F/x}f\|_{H^{s+1}(O)}\leq C_3\|I_W f\|_{H^{s+1}(\mathcal M_O)}.
\end{equation*}
This completes the proof of the local invertibility of $I_W$ on functions.
\end{proof}


\section{Local invertibility for the attenuated ray transform $I_{\mathcal A}$}\label{one form}

Now we consider the matrix weighted geodesic ray transform of combinations of 1-forms $\alpha$ and functions $f$ in the particular case that $W$ corresponds to the attenuation by a connection $A$ and a Higgs field $\Phi$
$$I_{\mathcal A}\begin{pmatrix} \alpha \\ f\end{pmatrix}(\gamma):=I_{\mathcal A}(\alpha+f)(\gamma)=\int W(\gamma,\dot{\gamma})(\alpha_{\gamma}(\dot{\gamma})+f(\gamma))\,dt,$$
where $\alpha(\dot{\g}(t))=\alpha_i(\g(t))\dot{\g}^i(t)$ in local coordinates. As we mentioned in the introduction,
$W$ is only continuous on $SO_p$ (smooth in $SO_p\backslash S(O_p\cap \p M)$), however this is easily fixed as follows.  By extending the metric and $(A,\Phi)$ to $\widetilde{M}$ (a neighbourhood of $M$) we can consider a new solution $\widetilde{W}$ to the transport equation \eqref{transport 1} in $\widetilde{M}$. Its restriction to $SM$ will be smooth. Since $I_{W}(f)=0$ if and only if $I_{\widetilde{W}}(f)=0$ in what follows we may just proceed as if $W$ was smooth.

As before, we denote $\Omega=\{x\geq 0\}$. Defining
$$Q:=\begin{pmatrix} Id_{N\times N} & 0\\ 0 & x^{-1}Id_{N\times N}\end{pmatrix},$$
we consider the following conjugated operator with $\chi$ the same as in Section \ref{function} and even
\begin{equation*}
\begin{split}
& N_F\begin{pmatrix} \alpha \\ f\end{pmatrix}(x,y)\\
=& Q^{-1}e^{-F/x}\int_{\R}\int_{\S^{n-2}}\begin{pmatrix} g_{sc}(\lambda\p_x+\omega\p_y)Id_{N\times N} \\ x^{-2} Id_{N\times N}\end{pmatrix} W^*\chi(\lambda/x) I_{\mathcal A} e^{F/x}Q\begin{pmatrix} \alpha \\ f\end{pmatrix} (x,y,\lambda,\omega)\, d\lambda d\omega\\
=& \begin{pmatrix} N_{11} & N_{10}\\ N_{01} & N_{00}\end{pmatrix} \begin{pmatrix} \alpha \\ f\end{pmatrix} (x,y),   
\end{split}
\end{equation*}
with
\begin{equation*}
\begin{split}
N_{11}\alpha(x,y)&=e^{-F/x}\int\int g_{sc}(\lambda\p_x+\omega\p_y)  W^*\chi(\lambda/x) (I_{\mathcal A} e^{F/x} \alpha)  (x,y,\lambda,\omega)\, d\lambda d\omega,\\
N_{10}f(x,y)&=e^{-F/x}\int\int g_{sc}(\lambda\p_x+\omega\p_y)  W^*\chi(\lambda/x) (I_{\mathcal A} e^{F/x} x^{-1}f)  (x,y,\lambda,\omega)\, d\lambda d\omega,\\
N_{01}\alpha (x,y)&=x^{-1}e^{-F/x}\int\int  W^*\chi(\lambda/x) (I_{\mathcal A} e^{F/x} \alpha)  (x,y,\lambda,\omega)\, d\lambda d\omega,\\
N_{00}f (x,y)&=x^{-1}e^{-F/x}\int\int  W^*\chi(\lambda/x) (I_{\mathcal A} e^{F/x}x^{-1} f) (x,y,\lambda,\omega)\, d\lambda d\omega.
\end{split}
\end{equation*}
Here $g_{sc}$ is the scattering metric which locally has the form $x^{-4}dx^2+x^{-2}h(x,y)$ where $h$ is a metric on the level sets of $x$. 

Similar to the analysis of \cite[Section 3]{SUV14} and the arguments for Lemma \ref{boundary symbol for function}, applying \eqref{asymptic2} one can get
\begin{lemma}\label{kernel}
The Schwartz kernel of $N_F$ in the interior of the scattering front face $x=0$ has the form 
\begin{equation*}
K_F^b=e^{-FX}|Y|^{-n+1}\chi(S)\begin{pmatrix} K_{11} & K_{10} \\ K_{01} & K_{00}\end{pmatrix}
\end{equation*}
with
\begin{equation*}
\begin{split}
K_{11} &=(S\frac{dx}{x^2}+\hat Y\frac{dy}{x})(W^*W)(0,y,0,\hat Y)((S+2\kappa|Y|)(x^2\p_x)+\hat Y(x\p_y)),\\
K_{10} &=(S\frac{dx}{x^2}+\hat Y\frac{dy}{x})(W^*W)(0,y,0,\hat Y),\\
K_{01} &=(W^*W)(0,y,0,\hat Y)((S+2\kappa|Y|)(x^2\p_x)+\hat Y(x\p_y)),\\
K_{00} &=(W^*W)(0,y,0,\hat Y),
\end{split}
\end{equation*}
where $S=\frac{X-\kappa|Y|^2}{|Y|}$ with $\kappa>0$.
\end{lemma}
In the above lemma, we use $^{sc}T^*\Omega\times ^{sc}T\Omega$ for the endomorphism bundle over $^{sc}T^*\Omega$, thus each entry of $K_{11}$ (matrix-valued) is a smooth section of the endomorphism bundle over $^{sc}T^*\Omega$. 

Notice that each $N_{ij},\, i,j=0,1$ is a scattering pseudodifferential operator of order $(-1,0)$ under the coordinates $(x,y,\tilde S,\hat Y,|Y|)$, so is $N_F$. As mentioned in Section \ref{function}, the Schwartz kernel of $N_F$ is smooth in $(x,y)$ up to $x=0$. By choosing $c$ (notice that the definition of $\Omega$ depends on $c>0$) sufficiently small, we want to show that up to some proper gauge, the operator $N_F$ is fully elliptic in some neighborhood of $\overline O_p=\Omega\cap M=\{x\geq 0, \rho\geq 0\}$ in $\Omega$ with compact closure. 

We define
$$d_{\A}:=\begin{pmatrix} d+A\\ \Phi \end{pmatrix}.$$
The natural elements in the kernel of $I_{\A}$ are 
$$\{d_{\A}p:\,p\in C^{\infty}(M;\mathbb C^N),\,p|_{\p M}=0\}.$$
It is not difficult to see that $d_{\mathcal A}$ is elliptic with trivial kernel under the zero Dirichlet boundary condition. Consider $\delta_{\mathcal A}$ the adjoint of $d_{\mathcal A}$ under the scattering metric $g_{sc}$, so $\delta_{\mathcal A}=\begin{pmatrix}\delta+A^* & \Phi^*\end{pmatrix}$. We define the following Witten-type solenoidal gauge 
\begin{equation}\label{gauge}
e^{2F/x}\delta_{\mathcal A}Q^{-1}e^{-2F/x}Q^{-1}\begin{pmatrix} \alpha \\ f\end{pmatrix}=0.
\end{equation}
If we denote $e^{F/x}\delta_{\mathcal A}Q^{-1}e^{-F/x}$ by $\delta_{\mathcal A,F}$ and $Q^{-1}e^{-F/x}\begin{pmatrix} \alpha \\ f\end{pmatrix}$ by $\begin{pmatrix}\alpha \\ f\end{pmatrix}_F$, then \eqref{gauge} becomes
$$\delta_{\mathcal A,F}\begin{pmatrix}\alpha \\ f\end{pmatrix}_F=0.$$

Before moving to the proof of the ellipticity, we compute the principal symbols of the gauge term $\delta_{\mathcal A,F}$ and $d_{\mathcal A,F}$, here $d_{\mathcal A,F}=e^{-F/x}Q^{-1}d_{\mathcal A}e^{F/x}$. Let $\Diff^{1,0}_{sc}$ denote the space of scattering differential operators of order $(1,0)$, which has a local basis $\{x^2\p_x, x\p_y\}$. 

\begin{lemma}\label{gauge symbol}
The operator $d_{\mathcal A,F}\in \Diff^{1,0}_{sc}(\Omega; \Hom (\mathbb C^N,(^{sc}T^*\Omega)^N\otimes \mathbb C^N))$ has principal symbol
\begin{equation*}
\begin{pmatrix}
\xi+iF \\
\eta\otimes \\
0
\end{pmatrix};
\end{equation*}
The operator $\delta_{\mathcal A,F}\in \Diff^{1,0}_{sc}(\Omega; \Hom ((^{sc}T^*\Omega)^N\otimes \mathbb C^N, \mathbb C^N))$ has principal symbol
\begin{equation*}
\begin{pmatrix}
\xi-iF & \iota_{\eta} & 0
\end{pmatrix}.
\end{equation*}
\end{lemma}

\begin{proof}
Notice that under the scattering cotangent basis $\{\frac{dx}{x^2},\,\frac{dy}{x}\}$, we have $d=(x^2\p_x,\,x\p_y)$. Consider $D:=(x^2(-i\p_x),\,x(-i\p_y))=(x^2D_x,\,x D_y)$ and $A=x^2A_x\frac{dx}{x^2}+xA_y\frac{dy}{x}$, so
\begin{equation*}
\begin{split}
e^{-F/x}Q^{-1}\begin{pmatrix} D+A \\ \Phi\end{pmatrix}e^{F/x}& =\begin{pmatrix} x^2D_x+iF+x^2A_x\\ xD_y+xA_y \\ x\Phi \end{pmatrix}\\
& =\begin{pmatrix} x^2D_x+iF \\ xD_y \\ 0\end{pmatrix}+x\begin{pmatrix} xA_x \\ A_y \\ \Phi\end{pmatrix}.
\end{split} 
\end{equation*}
Thus $d_{\mathcal A,F}$ is a scattering differential operator of order $(1,0)$ with principal symbol
$$\begin{pmatrix} \xi+iF \\ \eta\otimes \\ 0\end{pmatrix}.$$
Now the principal symbol of $\delta_{\mathcal A,F}$ is given by the adjoint of the one of $d_{\mathcal A,F}$ under the inner product of $g_{sc}$, i.e. $\begin{pmatrix} \xi-iF & \iota_{\eta} & 0\end{pmatrix}$, which is of order $(1,0)$ too.
\end{proof}

We analyze the ellipticity of $N_F$ at finite points and infinity of $^{sc}T^*\Omega$ when restricted to the kernel of the principal symbol of $\delta_{\mathcal A,F}$.

\begin{prop}\label{fiber infinity}
$N_F$ is elliptic at fiber infinity of $^{sc}T^*\Omega$ when restricted to the kernel of the principal symbol of $\delta_{\mathcal A,F}$.
\end{prop}
\begin{proof}
This part is similar to the proof of \cite[Lemma 3.4]{SUV14} and Proposition \ref{interior elliptic}. 
By Lemma \ref{kernel} we get that the principal symbol is essentially the following matrix
\begin{equation}\label{one form standard symbol}
|\zeta|^{-1}\int_{\zeta^{\perp}\cap(\mathbb R\times \mathbb S^{n-2})}\chi(\tilde S)\begin{pmatrix}
\tilde S \\
\hat Y \\
Id
\end{pmatrix}(W^*W)(0,y,0,\hat Y)\begin{pmatrix}\tilde S & \<\hat Y,\cdot\> & Id \end{pmatrix}\,d\tilde S d\hat Y,
\end{equation}
with $\tilde S=X/|Y|$. 
We want to show that \eqref{one form standard symbol} is invertible (positive definite) acting on the kernel of the principal symbol of $\delta_{\mathcal A,F}$.

Now given a non-zero $(v,f)$, with $v=(v^0,v')$, in the kernel of the standard principal symbol of $\delta_{\mathcal A,F}$, i.e. $v^0_j\xi+v'_j\cdot\eta=0$ for $j=1,\cdots,N$, we need to show that 
\begin{equation*}
\begin{split}
&\Big(\sigma_p(N_F)(y,\xi,\eta)\begin{pmatrix} v^0 \\ v' \\ f\end{pmatrix},\,\begin{pmatrix} v^0 \\ v' \\ f\end{pmatrix}\Big)\\
=& |\zeta|^{-1}\int_{\zeta^{\perp}\cap(\mathbb R\times \mathbb S^{n-2})}\chi(\tilde S)\Big|W(0,y,0,\hat Y)(\tilde S v^0+\hat Y\cdot v'+f)\Big|^2\,d\tilde S d\hat Y>0.
\end{split}
\end{equation*}
Note that 
\begin{equation*}
(\tilde S v^0+\hat Y\cdot v'+f)=\begin{pmatrix}
\tilde Sv^0_1+\hat Y\cdot v'_1+f_1 \\
\vdots\\
\tilde Sv^0_N+\hat Y\cdot v'_N+f_N
\end{pmatrix}.
\end{equation*}
Since $W$ is invertible, it suffices to show that there exists $(\tilde S,\hat Y)$ with $\chi(\tilde S)>0$ and $\xi\tilde S+\eta\cdot\hat Y=0$, such that $\tilde S v^0+\hat Y\cdot v'+f\neq 0$. 

We argue by contradiction.  Assume that $\tilde S v^0+\hat Y\cdot v'+f=0$ for all such $(\tilde S, \hat Y)$. Since $\chi$ is even, we get $-\tilde S v^0-\hat Y\cdot v'+f=0$ too, which implies that $\tilde S v^0+\hat Y\cdot v'=0$ and $f=0$. On the other hand, since $f=0$, we must have $(v^0,v')$ non-zero and thus there exists some $1\leq j\leq N$ with $(v^0_j,v'_j)$ is non-zero. Since the dimension $n$ is at least $3$ and $(\xi,\eta)\neq 0$, we can find $n-1$ linearly independent vectors from the set $\{(\tilde S, \hat Y): \xi \tilde S+\eta\cdot \hat Y=0, \, \chi(\tilde S)>0\}$ such that $\tilde S v^0_j+\hat Y\cdot v'_j=0$ for these vectors. Notice that the above set might be empty if $n=2$. Now taking into account the fact $\xi v^0_j+\eta\cdot v'_j=0$, by linear algebra, this implies that $(v^0_j,v'_j)=0$, which is a contradiction.

Thus the principal symbol \eqref{one form standard symbol} is positive definite, i.e. $N_F$ is elliptic at fiber infinity when restricted on the kernel of the principal symbol of $\delta_{\mathcal A,F}$.
\end{proof}

\begin{prop}\label{finite points}
$N_F$ is elliptic at finite points of $^{sc}T^*\Omega$ when restricted to the kernel of the principal symbol of $\delta_{\mathcal A,F}$.
\end{prop}

\begin{proof}
This part is similar to the proof of \cite[Lemma 3.5]{SUV14} and Proposition \ref{boundary elliptic}. Similar to the case of functions, we will start with $\chi$ being a Gaussian function.

Let $\chi(s)=e^{-s^2/(2F^{-1}\kappa)}$, following the calculation in \cite{SUV14} and the function case above, we get the $(X,Y)$-Fourier transform of $K^b_F$, $\sigma_{sc}(N_F)$, is a non-zero multiple of
\begin{equation}\label{one form scattering symbol}
\begin{split}
(\xi^2& + F^2)^{-1/2}\times\\
&\int_{\mathbb S^{n-2}}\begin{pmatrix}
-\overline{\beta}(\hat Y\cdot \eta) \\
\hat Y\\
Id
\end{pmatrix}W^*W(0,y,0,\hat Y)
\begin{pmatrix} -\beta(\hat Y\cdot\eta) & \<\hat Y,\cdot\> & Id
\end{pmatrix}e^{-\frac{|\eta\cdot\hat Y|^2}{2F^{-1}(\xi^2+F^2)\kappa}}\,d\hat Y,
\end{split}
\end{equation}
where $\beta=\frac{\xi-iF}{\xi^2+F^2}$.

Given a non-zero $(v^0,v',f)$ in the kernel of the scattering principal symbol of $\delta_{\mathcal A,F}$, i.e. $v^0_j(\xi-iF)+v'_j\cdot\eta=0$ for $j=1,\cdots,N$, we want to show that 
\begin{equation*}
\begin{split}
&\Big(\sigma_{sc}(N_F)(y,\xi,\eta)\begin{pmatrix} v^0 \\ v' \\ f\end{pmatrix},\,\begin{pmatrix} v^0 \\ v' \\ f\end{pmatrix}\Big)\\
=& c\, (\xi^2 + F^2)^{-1/2}\times\\
 & \quad \int_{\mathbb S^{n-2}}\Big|W(0,y,0,\hat Y)(-\beta(\hat Y\cdot\eta) v^0+\hat Y\cdot v'+f)\Big|^2e^{-\frac{|\eta\cdot\hat Y|^2}{2F^{-1}(\xi^2+F^2)\kappa}}\,d\hat Y>0.
\end{split}
\end{equation*}
Now the argument is similar to that of Proposition \ref{fiber infinity}. We only need to show that there exists $\hat Y\in \mathbb S^{n-2}$ such that $-\beta(\hat Y\cdot\eta) v^0+\hat Y\cdot v'+f\neq 0$. Assume that $-\beta(\hat Y\cdot\eta) v^0+\hat Y\cdot v'+f=0$ for all $\hat Y$. Then $\beta(\hat Y\cdot\eta) v^0-\hat Y\cdot v'+f=0$ too, which implies that $-\beta(\hat Y\cdot\eta) v^0+\hat Y\cdot v'=0$ and $f=0$. Now since $(v^0,v',f)\neq 0$ but $f=0$, we see that $(v^0,v')$ is non-zero and thus there exists some $1\leq j\leq N$ with $(v^0_j,v'_j)$ non-zero. Since $v^0_j(\xi-iF)+v'_j\cdot \eta=0$ and $F\neq 0$,
\begin{equation}\label{new equation}
0=-\beta (\hat Y\cdot \eta)v^0_j+\hat Y\cdot v'_j=\frac{1}{\xi^2+F^2}(\eta\cdot v'_j)(\hat Y\cdot \eta)+\hat Y\cdot v'_j
\end{equation}
for all $\hat Y$. Let $\hat Y\perp \eta$, \eqref{new equation} implies $\hat Y\cdot v'_j=0$, i.e. $v'_j=c\eta$ for some $c\in\mathbb R$. If $\eta=0$, then $v'_j=0$ too. If $\eta\neq 0$, let $\hat Y=\eta/|\eta|$, then \eqref{new equation} implies $c\frac{\xi^2+F^2+|\eta|^2}{\xi^2+F^2}|\eta|=0$. Now both $\frac{\xi^2+F^2+|\eta|^2}{\xi^2+F^2}$ and $|\eta|$ are non-zero, we conclude that $c=0$, i.e. $v'_j=0$. Consequently $v^0_j=-(\xi-iF)^{-1}(v'_j\cdot\eta)=0$, so $(v^0_j,v'_j)=0$, which is a contradiction.
This implies that the scattering principal symbol \eqref{one form scattering symbol} is positive definite when restricted to the kernel of the principal symbol of $\delta_{\mathcal A,F}$.

Moreover, by applying an argument similar to Proposition \ref{boundary elliptic}, it is not difficult to see that 
$$\Big(\sigma_{sc}(N_F)(y,\xi,\eta)\begin{pmatrix}v^0 \\ v' \\ f\end{pmatrix},\,\begin{pmatrix}v^0 \\ v' \\ f\end{pmatrix}\Big)\geq C\<(\xi,\eta)\>^{-1}\Big|\begin{pmatrix}v^0 \\ v' \\f\end{pmatrix}\Big|^2,$$
for some $C>0$. This is the desired bound for ellipticity.

Now by a similar approximation argument, we can find a smooth, even $\chi$ with compact support such that for this specific $\chi$ the operator $N_F$ is elliptic at finite points of the bundle $^{sc}T^*\Omega$ when restricted to the kernel of the principal symbol of $\delta_{\mathcal A,F}$.
\end{proof}

We combine Propositions \ref{fiber infinity} and \ref{finite points} to achieve the following ellipticity result, its proof is quite straightforward.

\begin{prop}\label{one form elliptic}
There exist a neighborhood $\tilde O$ of $\overline O_p$ in $\Omega$, a cutoff function $\chi\in C^{\infty}_c(\mathbb R)$ and an operator $P\in \Psi^{-3,0}_{sc}(\Omega)$ such that
$$N_F+d_{\mathcal A,F}P\delta_{\mathcal A,F}\in \Psi^{-1,0}_{sc}(\Omega; (^{sc}T^*\Omega)^N\otimes \mathbb C^N, (^{sc}T^*\Omega)^N\otimes \mathbb C^N)$$ 
is elliptic in $\tilde O$.
\end{prop}
\begin{proof}
Note that in Propositions \ref{fiber infinity} and \ref{finite points}, the analysis of the principal symbol of $N_F$ is at $\p \Omega=\{x=0\}$, thus by the smooth dependence on $(x,y)$ the ellipticity of $N_F$ (up to the solenoidal gauge) holds in some small neighborhood $\tilde O$ of $\p \Omega\cap \p M$ in $\Omega$, with compact closure. In particular by taking $c>0$ sufficiently small, we can make $\tilde O$ contain $\overline O_p$.

Let $P$ be a scattering pseudodifferential operator with the scattering principal symbol $(\xi^2+F^2+|\eta|^2)^{-3/2}$, thus $P$ has positive principal symbol and $P\in \Psi^{-3,0}_{sc}(\Omega)$. We use $\sigma_{\star}$ to represent both the standard and scattering principal symbols, given any $h=(v^0, v', f)$, assume that
$$\Big(\sigma_\star (N_F+d_{\mathcal A,F} P \delta_{\mathcal A,F}) h, h\Big)=(\sigma_\star(N_F) h, h)+\sigma_\star(P)|\sigma_\star(\delta_{\mathcal A,F}) h|^2=0.$$
Since $(\sigma_\star(N_F) h,h)$ is non-negative, we get that $(\sigma_\star(N_F) h,h)=0$ and $\sigma_\star(\delta_{\mathcal A,F})h=0$.  Now the proof of Propositions \ref{fiber infinity} and \ref{finite points} implies $h=0$. Therefore $N_F+d_{\mathcal A,F}P\delta_{\mathcal A,F}$ is elliptic.
\end{proof}

Similar to Theorem \ref{local function}, the ellipticity result of Proposition \ref{one form elliptic} gives the following stable invertibility of $I_{\mathcal A}$ under the solenoidal gauge. Its proof is similar to the one of Theorem \ref{local function}, so we omit it.
\begin{cor}\label{injective with gauge}
For sufficiently small $c>0$ and any $F>0$, $s\geq 0$, given $h=\alpha+f\in H^s(TO_p;\mathbb C^N) \oplus H^s(O_p;\mathbb C^N)$ with $\alpha$ a 1-form, if $\delta_{\mathcal A,F} h_F=0$ in $O_p$, then 
the $H^{s-1}$ norm of $h$ restricted to any compact subset of $O_p$ is controlled by the $H^s$ norm of $I_{\mathcal A}h|_{\mathcal M_{O_p}}$.
\end{cor}

Finally we need to rephrase the above invertibility result in a gauge free way, the argument is a direct modification of the one for the standard tensor tomography from \cite[section 4]{SUV14}. Recalling from the argument there, one needs to consider the `solenoidal Witten Laplacian' $\Delta_{\mathcal A,F}=\delta_{\mathcal A,F}d_{\mathcal A,F}$ on various regions $U$ in $\Omega$ with $O_p\subset \overline U$, $\p U=\p\Omega\cup \p_{int}U$, where $\p_{int}U$ is the part of the boundary of $U$ contained in the interior of $\Omega$, i.e. the interior boundary, so $U$ are manifolds with corners. In particular, $\p_{int}U$ could be $\p M$ or another further away hypersurface in $\Omega\backslash M$.

If $\Delta_{\mathcal A,F}$ is invertible with Dirichlet boundary condition imposed on $\p_{int} U$, we can decompose $h_F:=\begin{pmatrix}\alpha \\ f\end{pmatrix}_F$ into
$$h_F=\mathcal S_{\mathcal A,F}h_F+\mathcal P_{\mathcal A,F}h_F,$$
where $\mathcal P_{\mathcal A,F}=d_{\mathcal A,F}\Delta_{\mathcal A,F}^{-1}\delta_{\mathcal A,F}$. Thus we denote $\mathcal P_{\mathcal A,F}h_F$ by $d_{\mathcal A,F}p_F=Q^{-1}e^{-F/x}d_{\mathcal A}p$ with $p|_{\p O_p\cap \p M}=0$, then given $h=\begin{pmatrix} \alpha \\ f\end{pmatrix}$
$$I_{\mathcal A}h=I_{\mathcal A}(e^{F/x}Qh_F)=I_{\mathcal A}(e^{F/x}Q(h_F-d_{\mathcal A,F}p_F))=I_{\mathcal A}(e^{F/x}Q\mathcal S_{\mathcal A,F}h_F).$$
Notice that $\delta_{\mathcal A,F}(\mathcal S_{\mathcal A,F}h_F)=0$, by Corollary \ref{injective with gauge} if $I_{\mathcal A}h=0$ we have $\mathcal S_{\mathcal A,F}h_F=0$, i.e. $h_F=d_{\mathcal A,F}p_F$ or $h=d_{\mathcal A}p$ for some function $p$ with $p|_{\p O_p\cap \p M}=0$ which is exactly the natural kernel of $I_{\mathcal A}$. 
So one just needs to show that $\Delta_{\mathcal A,F}$ is invertible with Dirichlet boundary condition imposed on the interior boundary, however this is immediate from the argument of \cite[Section 4]{SUV14}. Note that as stated in Lemma \ref{gauge symbol}, $d_{\mathcal A,F}$, $\delta_{\mathcal A,F}$ has similar principal symbols as the ones of $d^s_F$ and $\delta^s_F$ in \cite{SUV14} respectively, thus $\Delta_{\mathcal A,F}$ and $\Delta_F=\delta^s_F d^s_F$ have the same principal symbol. It was already shown in \cite{SUV14} that $\Delta_F$ is invertible. Moreover the stable determination arguments of \cite[Section 4]{SUV14} work in a similar way for our X-ray transform $I_{\mathcal A}$, i.e. $h-d_{\mathcal A}p$ is stably controlled by $I_{\mathcal A}h$ on any compact subsets of $O_p$, with a reconstruction formula, this completes the proof of Theorem \ref{local connection}.
\qed


\section{Proof of the global result for $I_{\mathcal A}$} \label{sec_proof_global}

We will now give the proof of Theorem \ref{thm:maingloballinear}. If $U$ is a connected open subset of $M$ and $f: U \to \mathbb R$ is a smooth strictly convex exhaustion function, we will use the notation  
\begin{align*}
U_{\geq t} &= \{ x \in U \,;\, f(x) \geq t \}, \\
U_{> t} &= \{ x \in U \,;\, f(x) > t \}, \ \mathrm{etc}.
\end{align*}
The following simple lemma will be useful in gluing local constructions to global ones.

\begin{lemma} \label{lemma_exhaustion_boundary_geodesic}
Let $(M,g)$ be compact with strictly convex boundary, let $U \subset M$ be a connected open set, and let $f: U \to \mathbb R$ be a smooth strictly convex exhaustion function. Given $x_0 \in U$, any geodesic $\gamma$ in $M$ satisfying $\gamma(0) = x_0$ and $df(\dot{\gamma}(0)) \geq 0$ stays in the set $U_{\geq f(x_0)}$ and reaches $\partial M$ in finite time.
\end{lemma}
\begin{proof}
Let $x_0 \in U$, let $v_0 \in S_{x_0} M$ satisfy $df(v_0) \geq 0$, and let $\gamma: [0,T] \to M$ be the geodesic with $\gamma(0) = x_0$ and $\dot{\gamma}(0) = v_0$ where $T$ is the first time when $\gamma$ reaches the boundary (if it does not, set $T = \infty$). Then $h(t) = f(\gamma(t))$ satisfies 
\[
h(t) = h(0) + h'(0) t + \frac{1}{2} h''(c(t)) t^2
\]
where $c(t) \in [0,t]$. Since $h'(0) \geq 0$ and $h'' > 0$, $h$ is strictly increasing. Therefore $f(\gamma(t))\geq f(x_0)$ for any $t\in [0,T]$, which implies that $\gamma(t)$ stays in the set $U_{\geq f(x_0)}$.

It remains to show that $T < \infty$. If $f(x_0) > \inf f$ the set $U_{\geq f(x_0)}$ is compact and by strict convexity $h(t) \geq h(0) + ct^2$ for some $c > 0$, showing that $\gamma$ reaches the boundary in finite time (otherwise $f$ would not be bounded from above, contradicting Lemma \ref{lemma_exhaustion_function_properties}). On the other hand, if $f(x_0) = \inf f$ then Lemma \ref{lemma_exhaustion_function_properties} shows that $U = M$, $f$ is strictly convex in $M$, and any geodesic reaches the boundary in finite time.
\end{proof}

We will next establish the following result. It only differs from Theorem \ref{thm:maingloballinear} by recovering $h$ up to gauge in the set $U_{> a}$, which may be slightly smaller than $U$.

\begin{prop}
Let $(M,g)$ be a compact manifold with strictly convex boundary and $\dim(M) \geq 3$, let $U$ be a connected open subset of $M$, and let $f: U \to \mathbb R$ be a smooth strictly convex exhaustion function with $a = \inf_U f$. Let $(A,\Phi)$ be a $GL(N,\mathbb C)$-pair in $U$ and suppose $h(x,v)=f(x)+\alpha_{x}(v)$ where
$f\in C^{\infty}(U,\mathbb C^N)$ and $\alpha$ is a smooth $\mathbb C^N$-valued 1-form in $U$. If 
\[
(I_{\mathcal A} h)(\gamma) = 0 \ \ \text{for any geodesic $\gamma$ in $U$ with endpoints on $\partial M$},
\]
then 
\[
f = \Phi p \quad \text{and} \quad \alpha = dp + Ap \quad \text{in $U_{> a}$}
\]
for some $p \in C^{\infty}(U_{> a}, \mathbb C^N)$ with $p|_{\partial M} = 0$.
\label{prop:globallinear}
\end{prop}
\begin{proof}
By Lemma \ref{lemma_exhaustion_function_properties}, $f$ has no critical points in $U_{> a}$. Let $b = \sup_U f$ and consider the set 
\begin{align*}
I = \{ t \in (a,b) \,;\, h|_{U_{\geq t}} = (X+A+\Phi)v_t & \text{ for some $v_t \in C^{\infty}(U_{\geq t}, \mathbb C^N)$} \\
 & \qquad \text{with $v_t|_{U_{\geq t} \cap \partial M} = 0$} \}.
\end{align*}
We will show that 
\begin{enumerate}
\item[(1)]
$(b-\varepsilon, b) \subset I$ for some $\varepsilon > 0$; \\[-8pt]
\item[(2)]
if $c \in I$, then $(c-\varepsilon, c) \subset I$ for some $\varepsilon > 0$; \\[-8pt]
\item[(3)] 
if $(c,b) \subset I$ for some $c > a$, then $c \in I$.
\end{enumerate}
These facts imply that $I = (a,b)$ by connectedness. It follows that for any $t \in (a,b)$ there is $v_t \in C^{\infty}(U_{\geq t}, \mathbb C^N)$ with $h|_{U_{\geq t}} = (X+A+\Phi)v_t$ and $v_t|_{U_{\geq t} \cap \partial M} = 0$. If $s \leq t$, then $v_s-v_t$ satisfies  
\[
(X+A+\Phi)(v_s-v_t) = 0 \text{ in $U_{\geq t}$}, \quad v_s-v_t|_{U_{\geq t} \cap \partial M} = 0.
\]
Lemma \ref{lemma_exhaustion_boundary_geodesic} ensures that for any $z \in U_{\geq t}$ there is a geodesic in $U_{\geq t}$ connecting $z$ to $\partial M$. On this geodesic the equation $(X+A+\Phi)(v_s-v_t) = 0$ becomes an ODE, and uniqueness of solutions implies that $(v_s-v_t)(z) = 0$. Thus $v_s = v_t$ in $U_{\geq t}$ whenever $s \leq t$. Defining a function $p \in C^{\infty}(U_{> a}, \mathbb C^N)$ by $p|_{U_{\geq t}} = v_t$ for $t > a$, it follows that 
\[
h = (d+A+\Phi) p \text{ in $U_{> a}$}, \quad p|_{U_{> a} \cap \partial M} = 0.
\]
It remains to show (1)--(3).

\bigskip

{\it Proof of (1).} We note that $f^{-1}(b) = U_{\geq b}$ is compact, nonempty by Lemma \ref{lemma_exhaustion_function_properties}, and $f^{-1}(b) \subset \partial M$ since $f$ cannot attain a maximum at an interior point by strict convexity. The local uniqueness result, Theorem \ref{local connection}, shows that for any $p \in f^{-1}(b)$ there exists a neighbourhood $O_p$ in $M$ and a function $v_p \in C^{\infty}(O_p)$ satisfying 
\[
h|_{O_p} = (X+A+\Phi) v_p, \quad v_p|_{O_p \cap \partial M} = 0.
\]
Possibly after shrinking the neighbourhoods, we may assume that each $O_p$ is of the form $\{ \gamma_{z}(t) \,;\, z \in \Gamma, t \in [0,\varepsilon) \}$ for some open $\Gamma \subset \partial M$ and some $\varepsilon > 0$ where $\gamma_z$ is the inner normal geodesic to $\partial M$ through $z$.

By compactness $f^{-1}(b) \subset V$ where $V = \cup_{j=1}^m O_{p_j}$ for some $p_1, \ldots, p_m \in f^{-1}(b)$. Now if $O_{p_j} \cap O_{p_k} \neq \emptyset$, then 
\[
(X+A+\Phi) (v_{p_j} - v_{p_k}) = 0 \ \ \text{in $O_{p_j} \cap O_{p_k}$}, \quad v_{p_j} - v_{p_k}|_{(O_{p_j} \cap O_{p_k}) \cap \partial M} = 0.
\]
Evaluating the equation along normal geodesics shows that $v_{p_j} = v_{p_k}$ in $O_{p_j} \cap O_{p_k}$. Thus there exists $v \in C^{\infty}(V)$ with 
\[
h|_{V} = (X+A+\Phi) v, \quad v|_{V \cap \partial M} = 0.
\]
Since $U_{\geq b-\varepsilon} \subset V$ for some $\varepsilon > 0$, we have $(b-\varepsilon, b) \in I$. This proves (1).

We observe that the argument proving (1) also gives the following stronger result (using that $U_{\geq t} \cap \partial M$ is compact):
\begin{enumerate}
\item[(1')] 
For any $t > a$ there exists a neighbourhood $V$ of $U_{\geq t} \cap \partial M$ in $U$ and a function $v \in C^{\infty}(V, \mathbb C^N)$ such that 
\[
h|_V = (X+A+\Phi) v, \quad v|_{V \cap \partial M} = 0.
\]
\end{enumerate}

\medskip

{\it Proof of (2).} Let $c \in I$, which implies that there is $\tilde{v} \in C^{\infty}(U_{\geq c}, \mathbb C^N)$ so that 
\[
h|_{U_{\geq c}} = (X+A+\Phi)\tilde{v}, \quad \tilde{v}|_{U_{\geq c} \cap \partial M} = 0.
\]
By (1'), for any $p \in f^{-1}(c) \cap \partial M$ there exists a neighborhood $O_p$ and $v_p \in C^{\infty}(O_p, \mathbb C^N)$ so that 
\[
h|_{O_p} = (X+A+\Phi) v_p, \quad v_p|_{O_p \cap \partial M} = 0.
\]
We wish to glue the functions $\tilde{v}$ and $v_p$ together. Since $(d+A+\Phi)(\tilde{v}-v_p) = 0$ in $O_p \cap U_{\geq c}$ with zero boundary values on $\partial M$, one has $\tilde{v}-v_p|_{O_p \cap U_{\geq c}} = 0$ provided that:
\begin{equation} \label{op_u_connected}
\text{For any $x \in O_p \cap U_{\geq c}$, some smooth curve in $O_p \cap U_{\geq c}$ connects $x$ to $\partial M$.}
\end{equation}
To arrange this, extend $(M,g)$ and $f$ smoothly near $p$, let $(x',x_n)$ be semigeodesic coordinates for $f^{-1}(c)$ near $p$, and define 
\[
O_p = \{ (x',x_n) \,;\, \abs{x'} < \delta, \abs{x_n} < \delta \} \cap M.
\]
If $\delta > 0$ is small enough, the set 
\[
O_p \cap U_{\geq c} = \{ (x',x_n) \,;\, \abs{x'} < \delta, 0 \leq x_n < \delta \} \cap M
\]
will be connected, showing that \eqref{op_u_connected} holds. Thus we can glue the functions $\tilde{v}$ and $v_p$ to obtain a smooth extension, also denoted by $\tilde{v}$, to $U_{\geq c} \cup \overline{V}$ where $V$ is some neighborhood of $f^{-1}(c) \cap \partial M$, so that 
\[
h|_{U_{\geq c} \cup \overline{V}} = (X+A+\Phi)\tilde{v}, \quad \tilde{v}|_{(U_{\geq c} \cup \overline{V}) \cap \partial M} = 0.
\]

Now we consider the case where $p \in f^{-1}(c) \cap M^{\mathrm{int}}$. For any such $p$, the manifold $U_{\leq c}$ has a smooth and strictly convex boundary near $p$ since $df(p) \neq 0$. If $\eta$ is any short, close to tangential, geodesic in $U_{\leq c}$ near $p$, then strict convexity and Lemma \ref{lemma_exhaustion_boundary_geodesic} show that $\eta$ can be uniquely extended as a geodesic $\gamma: [0,T] \to M$ so that, for some $t_1, t_2$ with $0 < t_1 < t_2 < T$, 
\begin{align*}
 &\text{$\gamma|_{[0,t_1]}$ goes from $\partial M$ to $f^{-1}(c)$ in $U_{\geq c}$}, \\
 &\text{$\gamma|_{[t_1,t_2]}$ is a reparametrization of the short geodesic $\eta$ in $U_{\leq c}$, and } \\
 &\text{$\gamma|_{[t_2,T]}$ goes from $f^{-1}(c)$ to $\partial M$ in $U_{\geq c}$}.
\end{align*}
Since $\gamma$ stays in $U$, we have $(I_{\mathcal A} h)(\gamma) = 0$ by assumption. Choose an arbitrary smooth extension of $\tilde{v}$ to $M$, and also denote this extension by $\tilde{v}$. Then trivially $(I_{\mathcal A} ((X+A+\Phi)\tilde{v}))(\gamma) = 0$ for all such $\gamma$. It follows that 
\[
(I_{\mathcal A} (h - (X+A+\Phi)\tilde{v}))(\gamma) = 0
\]
for any $\gamma$ generated by a short geodesic $\eta$ in $U_{\leq c}$ near $p$. Since $h - (X+A+\Phi)\tilde{v}|_{U_{\geq c}} = 0$, this implies that 
\[
(I_{\mathcal A}^{(c)} (h - (X+A+\Phi)\tilde{v}|_{U_{\leq c}}))(\eta) = 0
\]
for any short geodesic $\eta$ near $p$, where $I_{\mathcal A}^{(c)}$ is the attenuated ray transform in $U_{\leq c}$ near $p$. The conditions of the local result, Theorem \ref{local connection}, are satisfied. Thus the local result implies that for any $p \in f^{-1}(c) \cap M^{\mathrm{int}}$ there exists a neighborhood $O_p$ in $U_{\leq c}$ and a function $v_p \in C^{\infty}(O_p)$ such that 
\[
h - (X+A+\Phi)\tilde{v}|_{O_p} = (X+A+\Phi) v_p, \quad v_p|_{O_p \cap f^{-1}(c)} = 0.
\]
If $(x',x_n)$ are semigeodesic coordinates for $f^{-1}(c)$ near $p$, after possibly shrinking $O_p$ we may assume that $O_p$ is of the form $\{ (x',x_n) \,;\, \abs{x'} < \delta, \ -\delta < x_n \leq 0 \}$ for some $\delta > 0$.

Write $w_p = \tilde{v} + v_p$ in $O_p$, and consider the function 
\[
\hat{v} = \left\{ \begin{array}{cl} \tilde{v} & \text{in $U_{\geq c} \cup \overline{V}$}, \\ w_p & \text{in $O_p$ with $p \in f^{-1}(c) \cap M^{\mathrm{int}}$}. \end{array} \right.
\]
If we can show that $\hat{v}$ is well defined in the overlaps and smooth, (2) will follow since then $\hat{v}$ is a smooth function in some neighborhood $W$ of $U_{\geq c}$ in $U_{> a}$ and satisfies 
\[
h|_W = (X+A+\Phi)\hat{v}, \quad \hat{v}|_{W \cap \partial M} = 0.
\]
Let $p \in f^{-1}(c) \cap M^{\mathrm{int}}$. Writing $(x',x_n)$ for semigeodesic coordinates for $f^{-1}(c)$ near $p$, it follows that 
\begin{align*}
\partial_{x_n} w_p + (A(\partial_{x_n}) + \Phi) w_p &= h(x,\partial_{x_n}), \qquad \text{$x_n < 0$}, \\
\partial_{x_n} \tilde{v} + (A(\partial_{x_n}) + \Phi) \tilde{v} &= h(x,\partial_{x_n}), \qquad \text{$x_n > 0$}, \\
w_p|_{x_n=0} &= \tilde{v}|_{x_n = 0}.
\end{align*}
Inductively we see that $\partial_{x_n}^j w_p|_{x_n=0} = \partial_{x_n}^j \tilde{v}|_{x_n=0}$ for all $j \geq 0$. This shows that $\hat{v}$ will be smooth across $f^{-1}(c)$ away from $\partial M$. It remains to show that if $p \in f^{-1}(c) \cap M^{\mathrm{int}} \cap \overline{V}$, then $w_p = \tilde{v}$ in $O_p \cap \overline{V}$. But since $O_p$ is of the form $\{ \abs{x'} < \delta, \ -\delta < x_n \leq 0 \}$ in semigeodesic coordinates and since $V$ was constructed as described after \eqref{op_u_connected}, we can connect any $x \in O_p \cap \overline{V}$ to $f^{-1}(c)$ by a smooth curve in $O_p \cap \overline{V}$. Thus $w_p$ and $\tilde{v}$ solve the same ODE along this curve with matching boundary values on $f^{-1}(c)$, which finally shows that $\hat{v}$ is well-defined and smooth. This concludes the proof of (2).

\bigskip

{\it Proof of (3).} Suppose that $(c,b) \subset I$ where $c > a$. The gluing procedure described after the conditions (1)--(3) ensures that there is $w \in C^{\infty}(U_{> c}, \mathbb C^N)$ so that 
\[
h = (d+A+\Phi) w \text{ in $U_{> c}$}, \quad w|_{U_{> c} \cap \partial M} = 0.
\]
It is enough to show that $w$ has an extension $\tilde{w} \in C^{\infty}(U_{\geq c}, \mathbb C^N)$. Denote by $W$ the  solution of 
\[
XW - W(A+\Phi) = 0 \text{ in $E_c$}, \quad W|_{E_c \cap \partial(SM)} = \mathrm{id}
\]
where $E_c = \{ (x,v) \in SM \,;\, x \in U_{\geq c}, \ df|_x(v) \geq 0 \}$. By Lemma \ref{lemma_exhaustion_boundary_geodesic}, $W$ is smooth in $E_c \setminus S(\partial M)$. Since $X(Ww) = W(X+A+\Phi)w = Wh$ in $SU_{> c} \cap E_c$, it follows that 
\[
w(x) = -W(x,v)^{-1} \int_0^{\tau(x,v)} (Wh)(\varphi_t(x,v)) \,dt, \quad x \in U_{> c}, \ df|_x(v) \geq 0.
\]
Denote by $\tilde{w}(x)$ the right hand side evaluated at $(x,v) = (x,Y(x))$ where $Y$ is a smooth vector field in $U_{\geq c}$ with $df(Y) \geq 0$ and $Y|_{\partial M}$ nontangential (one can take $Y = \nabla f/\abs{\nabla f}$ away from $\partial M$ and apply a partition of unity). Then $\tilde{w} \in C^{\infty}(U_{\geq c}, \mathbb C^N)$ satisfies $\tilde{w}|_{U_{> c}} = w$, giving the required extension.
\end{proof}

\begin{proof}[Proof of Theorem \ref{thm:maingloballinear}]
Let $U \subset M$ be a connected open set, let $f: U \to \mathbb R$ be a smooth strictly convex exhaustion function, and assume that $(I_{\mathcal A} h)(\gamma) = 0$ for any geodesic $\gamma$ in $U$ having endpoints on $\partial M$. Let $a = \inf f$ and $b = \sup f$.

Now Proposition \ref{prop:globallinear} proves the theorem if $U_{> a} = U$. By Lemma \ref{lemma_exhaustion_function_properties} the only other possible scenario is the case where $U = M$, $f$ is smooth and strictly convex in $M$, and $U_{> a} = M \setminus \{ z_0 \}$ where $z_0$ is the unique minimum point of $f$. Now $M$ is nontrapping by Lemma \ref{lemma_convex_anydimension}; since $\p M$ is strictly convex too, the fact that $I_{\mathcal A} h = 0$ and the regularity result \cite[Proposition 5.2]{PSU12} show that there exists $u\in C^{\infty}(SM)$ satisfying the  transport equation
\begin{equation}\label{transport equation on SM}
(X+A+\Phi)u=-h \ \ \text{in $SM$}, \quad u|_{\p SM}=0.
\end{equation}
Also by Proposition \ref{prop:globallinear} there is $p \in C^{\infty}(M \setminus \{ z_0 \})$ with 
\[
(X+A+\Phi)p = h \ \ \text{in $S(M \setminus \{ z_0 \})$}, \quad p|_{\partial M \setminus \{ z_0 \}} = 0.
\]
Thus $u+p$ satisfies 
\[
(X+A+\Phi)(u+p)=0\,\, \mbox{in}\,\, S(M\setminus \{z_0\}),\quad u+p|_{\p S(M \setminus \{z_0\})}=0.
\]
By Lemma \ref{lemma_exhaustion_boundary_geodesic}, for any $z \in M \setminus \{ z_0 \}$ there is a smooth geodesic in $M \setminus \{ z_0 \}$ connecting $z$ to the boundary. The equation for $u+p$ then implies that 
\[
u+p = 0 \text{ in $S(M \setminus \{ z_0 \})$}.
\]
Finally, since $p$ only depends on $x$, the same is true for $u$ in $S(M \setminus \{ z_0 \})$. Since $u$ is smooth in $SM$ we obtain that $u \in C^{\infty}(M)$, and returning to \eqref{transport equation on SM} gives that 
\[
h = (d+A+\Phi)\tilde{p} \ \ \text{ in $M$}, \ \ \tilde{p} \in C^{\infty}(M), \ \ \tilde{p}|_{\partial M} = 0
\]
with $\tilde{p} = -u$.
\end{proof}


\section{Nonlinear problem for connections and Higgs fields}\label{nonlinear connection}

Now we consider the nonlinear problem of determining a connection and a Higgs field from the corresponding scattering relation.
Theorem \ref{thm:mainglobal} follows from Theorem \ref{thm:maingloballinear} by introducing a pseudo-linearization. The argument
is carried out in detail in \cite[Section 8]{PSU12}. Instead of repeating it here, we shall explain how to use it to prove
the following nonlinear local problem.

\begin{thm}\label{scattering relation}
Assume $\p M$ is strictly convex at $p\in\p M$. Let $A$ and $B$ be two connections, let $\Phi$ and $\Psi$ be two Higgs fields, and write $\mathcal A=A+\Phi$ and $\mathcal B=B+\Psi$. If $C_{\mathcal A}=C_{\mathcal B}$ near $S_p\p M$, then there exist a neighborhood $O_p$ of $p$ in $M$ and a smooth $U: O_p\to GL(N,\mathbb C)$ with $U|_{O_p\cap \p M}={\rm id}$ such that $\mathcal B=U^{-1}dU+U^{-1}\mathcal AU$ in $O_p$.  
\end{thm}

\begin{proof}
Let $\mathcal A=A+\Phi$ and $\mathcal B=B+\Psi$ be two pairs of connections and Higgs fields, and consider the matrix weights $W_Q,\, Q=\mathcal A,\,\mathcal B$, which satisfy $XW_Q=W_Q Q$ with $W_Q|_{\p_+SM}={\rm id}$. We define $U=W_{\mathcal A}^{-1}W_{\mathcal B}$, it is easy to check that $U$ satisfies the following transport equation
\begin{equation}\label{transport}
\left\{ \begin{array}{l}
  XU+\A U-U\B=0, \\
  U|_{\p_+SM}={\rm id}, \\
  U|_{\p_-SM}=C_{\A}^{-1}C_{\B}.
  \end{array} \right.
\end{equation}
We rewrite the above transport equation as
\[
X(U-{\rm id})+\A (U-{\rm id})-(U-{\rm id})\B=\B-\A,
\]
and define a new connection $\hat A$ as $\hat A(W)=AW-WB$ and a new Higgs field $\hat{\Phi}$ as $\hat{\Phi}=\Phi W-W \Psi$ for any $W\in C^{\infty}(SM;\mathbb C^{N\times N})$. Then for the pair $\hat{\A}=\hat A+\hat{\Psi}$
\[
X(U-{\rm id})+\hat{\A}(U-{\rm id})=\B-\A.
\]
Now take into account that $C_{\A}=C_{\B}$ near $S_p\p M$, thus $(U-{\rm id})(x,v)=0$ for $(x,v)\in \p SM$ near $S_p\p M$. This implies that
\[
\int \hat W(\gamma(t),\dot{\gamma}(t))(\B-\A)(\gamma(t),\dot{\gamma}(t))\,dt=0
\]
for $\gamma\in \mathcal M_{O_p}$, where $\hat W$ is the matrix attenuation associated with $\hat{\A}$.

Note that the entries of $\B-\A$ are pairs of 1-forms and functions. Thus by Theorem \ref{local connection}, there exists $V\in C^{\infty}(O_p;\mathbb C^{N\times N})$ with $V|_{O_p\cap \p M}=0$, such that $\B-\A=dV+\hat{\A} V$ in $O_p$. This implies that 
\[
(X+\hat{\A})(U-{\rm id} - V) = 0 \text{ in $SO_p$}, \quad U-{\rm id} - V|_{SO_p \cap \partial M} = 0.
\]
Consequently $U-{\rm id}=V$ only depends on $x$ in $O_p$, and by \eqref{transport}, $\B=U^{-1}dU+U^{-1}\A U$ with $U|_{O_p\cap\p M}={\rm id}$.
\end{proof}


We can give an alternative proof of the previous theorem using a pseudo-linearization approach as in \cite{SU98, SUV13}, which ends up giving the same types of integrals. Given a geodesic $\gamma: [0,T]\to M$, let $\phi(t)=(\gamma(t),\dot{\gamma}(t))$ be the corresponding geodesic flow on $SM$. Define the matrix-valued function
\[
F(t)=W_{\B}(\phi(t))W_{\A}^{-1}(\phi(t)).
\]
Thus 
\[
F'(t)=W_{\B}(\phi(t))({\B}(\phi(t))-{\A}(\phi(t)))W_{\A}^{-1}(\phi(t)).
\]
By the fundamental theorem of calculus 
\[
F(T)-F(0)=\int_0^T W_{\B}(\phi(t))(\B (\phi(t))-\A (\phi(t)))W_{\A}^{-1}(\phi(t))\,dt.
\]
By the assumption for $\gamma\in \mathcal M_{O_p}$, $W_{\A}(\phi(T))=W_{\B}(\phi(T))$, i.e.\ $F(T)=F(0)={\rm id}$, so we get
\[
\int_{\gamma} W_{\B}(\phi(t))(\B (\phi(t))-\A (\phi(t)))W_{\A}^{-1}(\phi(t))\,dt=0.
\]
We define the linear transformation (matrix) $\hat W$ by 
\[
\hat W U=W_{\B}UW_{\A}^{-1}, \qquad U\in C^{\infty}(SM;\mathbb C^{N\times N}).
\]
Then $\hat W$ is smooth near $S_p\p M$ and the above integrals over $\gamma$ actually state that the following weighted geodesic ray transform of $\B-\A$ vanishes:
\[
\int_{\gamma} \hat W(\B-\A)\,dt=0,\quad \gamma\in \mathcal M_{O_p}.
\]
Now to prove Theorem \ref{scattering relation}, we only need to show that $\hat W$ is associated with some connection. First by the definition of $\hat W$, it is obvious that $\hat W|_{\p_+SM}={\rm id}$. Given a matrix-valued function $U$ we have 
\begin{equation*}
(X\hat W)U+\hat W(XU)=X(W_{\B}UW_{\A}^{-1})=W_{\B}\B UW_{\A}^{-1}+W_{\B}(XU)W_{\A}^{-1}-W_{\B}U \A W_{\A}^{-1},
\end{equation*}
which implies that
\[
(X\hat W)U=\hat W(\B U-U \A).
\]
If we define $\hat{\A}$ by $\hat{\A}U=\B U-U\A$, we get exactly $X\hat W=\hat W\hat {\A}$. Theorem \ref{local connection} then finishes the proof.


\section{Further applications}
\label{section:furtherapplications}

\subsection{Quantum state tomography}

First we discuss an application of our results to {\it quantum state tomography} \cite{Il15}. Roughly speaking, quantum state tomography is concerned with the determination of a quantum mechanical Hamiltonian $H$ (a smooth function defined on the phase space and taking values in $\mathbb C^{N\times N}$) from the knowledge of time evolution of coherent quantum mechanical states through the quantum mechanical system. It has potential applications in geophysical imaging with neutrinos. Let $(M,g)$ be a compact Riemannian manifold with smooth boundary and $\gamma:[0,T]\to M$ be a geodesic. A point particle moving along the geodesic $\gamma$ is associated with a quantum mechanical state $\Psi^{\g}(t)\in \mathbb C^N$. The time evolution of the states is governed by the Schr\"odinger equation
\[
i\p_t\Psi^{\g}(t)=H(\gamma(t),\dot{\gamma}(t))\Psi^{\g}(t).
\]
Then there is a time evolution operator (nondegenerate, matrix-valued) $U^{\gamma}_H$ associated with the Hamiltonian $H$ along $\gamma$ such that 
\[
\Psi^{\g}(t_2)=U^{\gamma}_H(t_2,t_1)\Psi^{\g}(t_1), \quad t_1,t_2\in [0,T],
\]
in particular $U^{\gamma}_H(t,t)={\rm id}_{N\times N}$. If $\gamma$ is a geodesic connecting boundary points, i.e.\ $\gamma(0),\gamma(T)\in \p M$, we take all possible initial states $\Psi^{\g}(0)$ in $\mathbb C^{N}$ and measure the corresponding final states $\Psi^{\g}(T)$, then one can uniquely determine the time evolution matrix $U^{\g}_H(T,0)$. Note that for this case, we can also define the time evolution operator as $$U_H(\gamma(t),\dot{\gamma}(t)):=U^{\gamma}_H(t,0),$$
and $U_H|_{\p_+SM}={\rm id}_{N\times N}$. 
For the sake of simplicity we assume that we have enough data and we can measure these quantum states individually to determine the time evolution operator. Our result is about the local unique recovery of the Hamiltonian $H$ from the time evolution operator $U_H$. We restrict ourselves to the case where $H$ is a combination of matrix-valued functions and 1-forms on $M$. 
\begin{thm}\label{quantum}
Assume that $\dim(M) \geq 3$ and $\p M$ is strictly convex at $p\in\p M$. Let $H_1=A+\Phi$ and $H_2=B+\Psi$, where $A,B\in C^{\infty}(TM;\mathbb C^{N\times N})$ are linear in $v$ and $\Phi,\Psi\in C^{\infty}(M;\mathbb C^{N\times N})$, be two Hamiltonians with $U_{H_1}(x,v)=U_{H_2}(x,v)$ for $(x,v)\in \p_-SM$ near $S_p\p M$. Then there is a neighborhood $O_p$ of $p$ in $M$ and a smooth $U: O_p\to GL(N,\mathbb C)$ with $U|_{O_p\cap \p M}=Id$ such that $H_2=U^{-1}dU+U^{-1}H_1U$ in $O_p$.
\end{thm}

The proof of Theorem \ref{quantum} is almost identical to that for Theorem \ref{scattering relation}. If one replaces $\A$ by a Hamiltonian $H$, then $U_H$ gives the corresponding ''scattering relation'' for $H$. Similarly a global determination result holds under the foliation condition.

A related more realistic problem was considered in \cite{Il15} for the case that $H$ is a function in a domain in Euclidean space. Our approach generalizes some results of \cite{Il15} to the Riemannian setting and to more general Hamiltonians.

\subsection{Polarization tomography}

Another application of our main results is related to {\it polarization tomography}. This inverse problem arises in optical tomography of slightly anisotropic media, and it consists of recovering the anisotropic part of a quasi-isotropic medium from polarization measurements made around the boundary. The anisotropic part is represented by a complex matrix function $f$. We consider the problem in the Riemannian setting, and $f\in T_{1,1}^{\mathbb C}M$ is a complexified $(1,1)$ tensor on a Riemannian manifold $(M,g)$ with boundary. Locally $f$ can be viewed as a complex matrix function whose size is $n\times n$ with $n=\mbox{dim}\, M$.

Now given $f\in T_{1,1}^{\mathbb C}M$, 
we consider the following equation of a complex vector field $\eta\in T^{\mathbb C}M$ along a maximal geodesic $\gamma:[0,T]\to M$, 
\begin{equation*}\label{polarization 1}
D_t \eta(t)=(P_{\gamma(t),\dot{\gamma}(t)}f)\eta(t),
\end{equation*}
where $D_t$ is the covariant derivative along $\gamma(t)$ and $P_{z,v}f=\pi_{z,v} f \pi_{z,v}$ with $\pi_{z,v}$ the orthogonal projection onto the subspace of $T_z^{\mathbb C}M$ perpendicular to $v$. Given the initial vector $\eta(0)=\eta_0$, we measure the value at the endpoint $\eta(T)$. The inverse problem is to recover $f$ from the polarization measurements $\eta(T)$ for all possible $\gamma$ and $\eta_0$.

Similarly as in quantum state tomography, we may reformulate the problem by introducing a time evolution operator $U$ on $SM$ which satisfies the transport equation
\begin{equation*}\label{polarization 2}
XU(z,v)=(P_{z,v}f)U(z,v), \quad U|_{\p_+SM}={\rm id}.
\end{equation*}
It is convenient to consider $U$ as a $(1,1)$ tensor field that depends on $(z,v)\in SM$, so locally 
$$U(z,v)=U^i_j(z,v)\p_{z^i}\otimes dz^j,$$
which can be viewed as a matrix function on $SM$. We call $U|_{\p_-SM}$ the polarization data, so the problem can be formulated as recovering $f$ from $U|_{\p_-SM}$. Notice that the above problem is also a nonlinear inverse problem, in particular the dependence of $P_{z,v}$ on $v$ is also nonlinear. This problem has been studied in \cite{NS07, Sh94, Ho13} by pseudo-linearizations as in Section \ref{nonlinear connection}. However, in dimension $3$ there is a natural obstruction to the unique determination of the global problem, see \cite{NS07}. Here we only consider the local problem in the case where $\dim(M) \geq 5$.

By pseudo-linearization, given $f_1,\, f_2\in T^{\mathbb C}_{1,1}M$ and the values of corresponding $U_1,\, U_2$ at $\p_-SM$, we have the following integral identity for any geodesic $\gamma=\gamma_{z,v}$ and $T\in [0,\tau(z,v)]$ with $(z,v)=(\gamma(0),\dot{\gamma}(0))\in \p_+SM$ and $\tau(z,v)$ the positive exit time of $\gamma_{z,v}$
\begin{align*}
U_2^{-1}U_1 & (\gamma(T), \dot{\gamma}(T))-{\rm id}\\
&=\int_0^T \Upsilon^{\gamma}_{t,T}\Big[U^{-1}_2(\gamma(t),\dot{\gamma}(t))(P_{\gamma(t),\dot{\gamma}(t)}(f_1-f_2))(\gamma(t))U_1(\gamma(t),\dot{\gamma}(t))\Big]\,dt
\end{align*}
where $\Upsilon^\gamma_{t_1,t_2}$ is the parallel transport of tensors along the geodesic $\gamma$ from $\gamma(t_1)$ to $\gamma(t_2)$.
This reduces the nonlinear problem to inverting the weighted geodesic ray transform 
\begin{equation*}
I_{U_1,\,U_2}f(\gamma):=\int_{\gamma}\Upsilon^\gamma_{t,\tau(z,v)}\Big[U^{-1}_2(\gamma(t),\dot{\gamma}(t))(P_{\gamma(t),\dot{\gamma}(t)}f)(\gamma(t))U_1(\gamma(t),\dot{\gamma}(t))\Big]\,dt,
\end{equation*}
see also \cite[Section 3]{NS07}. The microlocal properties of the normal operator $N_{U_1,\,U_2}=I^*_{U_1,\,U_2}\circ I_{U_1,\,U_2}$ were studied in \cite{Ho13}. On simple manifolds, $N_{U_1,\,U_2}$ is a pseudodifferential operator of order $-1$, in particular it is elliptic if dim $M\geq 4$. As in the nonlinear problem, $I_{U_1,\,U_2}$ has a nontrivial kernel in dimension 3. 

Define $W(z,v)=\Upsilon^{\gamma_{z,v}}_{0,\tau(z,v)}\Big[U^{-1}_2(z,v)(P_{z,v}\,\cdot\,)U_1(z,v)\Big]$ on $SM$, we consider the following operator similar to \eqref{N_F function} near a strictly convex boundary point:
\begin{equation*}
(N_Ff)(x,y)=x^{-2}e^{-F/x}\int_{\R}\int_{\S^{n-2}}W^*(x,y,\l,\o) (I_{U_1,\,U_2} e^{F/x}f)(\xy)\chi(\l/x)\, d\l d\o.
\end{equation*}
So $N_F$ is a scattering pseudodifferential operator of order $(-1,0)$. Recall the proof of Proposition \ref{interior elliptic} and \ref{boundary elliptic}, we want to prove the ellipticity of $N_F$, since $U_1$ and $U_2$ are invertible, it suffices to show that for an arbitrary covector $(\xi,\eta)$, non-zero $f$ and $v\in T_{x,y}^{\mathbb C}M$ with $f(v)\neq 0$ there exists some $\hat Y\in \eta^{\perp}\cap\mathbb S^{n-2}$ such that $(P_{x,y,0,\hat Y}f)v\neq 0$.

Now if the dimension satisfies $n\geq 5$, we can always find $\hat Y\in \eta^{\perp}\cap\mathbb S^{n-2}$ such that $(0,\hat Y)\perp\,\mbox{span}\,\{v, f(v)\}$ so
\[
(P_{x,y,0,\hat Y}f)v=\pi_{x,y,0,\hat Y}f\pi_{x,y,0,\hat Y}(v)=f(v)\neq 0,
\]
which proves the ellipticity. We thus obtain the following local result for polarization tomography.

\begin{thm}
Assume $\dim(M) \geq 5$ and $\p M$ is strictly convex at $p\in\p M$. Let $f_1$ and $f_2$ be complex matrix-valued functions on $M$ with $U_1(z,v)=U_2(z,v)$ for $(z,v)\in \p_-SM$ near $S_p\p M$. Then there exists a neighborhood $O_p$ of $p$ in $M$ such that $f_1=f_2$ in $O_p$.
\end{thm}

We can also prove a global theorem in this direction (we omit the proof).

\begin{thm} Let $(M,g)$ be a compact Riemannian manifold of dimension $\geq 5$ and strictly convex boundary.
 Suppose $(M,g)$ admits a smooth strictly convex function. Let $f_{1}$ and $f_2$ be two complex matrix-valued functions on $M$ with
 the same polarization data, i.e. $U_{1}(x,v)=U_{2}(x,v)$ for $(x,v)\in \partial_{-}(SM)$. Then $f_{1}=f_{2}$.
\end{thm}




\end{document}